\documentclass[reqno,11pt]{amsart}

\usepackage{graphics}
\usepackage{amssymb, amsmath}
\usepackage{version}
\usepackage{fancybox}
\usepackage{bm}
\usepackage{xcolor}
\usepackage{mathrsfs}
\usepackage{mathtools} % for \MoveEqLeft
\theoremstyle{theorem}
\newtheorem{theorem}{\sc Theorem}[section]
\newtheorem{lemma}[theorem]{\sc Lemma}

\newtheorem{proposition}[theorem]{\sc Proposition}
\newtheorem{corollary}[theorem]{\sc Corollary}

\theoremstyle{definition}
\newtheorem{definition}[theorem]{\sc Definition}
\newtheorem*{notation*}{Notation}

\theoremstyle{remark}
\newtheorem{remark}[theorem]{\sc Remark}

\makeatletter

\@addtoreset{equation}{section}

\renewcommand{\lefteqn}{\MoveEqLeft}

\renewcommand{\d}{\text{\rm d}}

\newcommand{\vep}{\varepsilon}
 
\newcommand{\e}{{\rm e}}

\newcommand{\R}{\mathbb{R}}
\newcommand{\N}{\mathbb{N}}

\newcommand{\A}{\mathcal{A}}
\newcommand{\B}{\mathcal{B}}

\newcommand{\J}{\mathcal{J}}

\DeclareMathOperator*{\esssup}{ess\,sup}

\makeatletter
\newcommand{\opnorm}{\@ifstar\@opnorms\@opnorm}
\newcommand{\@opnorms}[1]{%
  \left|\mkern-1.5mu\left|\mkern-1.5mu\left|
   #1
  \right|\mkern-1.5mu\right|\mkern-1.5mu\right|
}
\newcommand{\@opnorm}[2][]{%
  \mathopen{#1|\mkern-1.5mu#1|\mkern-1.5mu#1|}
  #2
  \mathclose{#1|\mkern-1.5mu#1|\mkern-1.5mu#1|}
}
\makeatother

\allowdisplaybreaks[4]

\newcommand{\prf}[1]{}  % without proof
\begin{document}
\title[Time-fractional gradient flows for nonconvex energies]{Time-fractional gradient flows for nonconvex energies in Hilbert spaces\prf{\\{\tt [Extended version]}}}
%
%\subtitle{}
%
\author{Goro Akagi}
\address[Goro Akagi]{Mathematical Institute and Graduate School of Science, Tohoku University, Aoba, Sendai 980-8578, Japan}
\email{goro.akagi@tohoku.ac.jp}
\author{Yoshihito Nakajima}
\address[Yoshihito Nakajima]{Graduate School of Science, Tohoku University, Aoba, Sendai 980-8578, Japan}
\email{yoshihito.nakajima.p8@dc.tohoku.ac.jp}
\thanks{G.A.~is supported by JSPS KAKENHI Grant Numbers JP24H00184, JP21KK0044, JP21K18581, JP20H01812 and JP20H00117. 
Y.N. is supported by JST SPRING, Grant Number JPMJSP2114. 
This work was also supported by the Research Institute for Mathematical Sciences, an International Joint Usage/Research Center located in Kyoto University.}
\date{\today}
\maketitle
\begin{abstract}
This article is devoted to presenting an abstract theory on time-fractional gradient flows for nonconvex energy functionals in Hilbert spaces. Main results consist of local and global in time existence of (continuous) strong solutions to time-fractional evolution equations governed by the difference of two subdifferential operators in Hilbert spaces. To prove these results, fractional chain-rule formulae, a Lipschitz perturbation theory for convex gradient flows and Gronwall-type lemmas for nonlinear Volterra integral inequalities are developed. They also play a crucial role to cope with the lack of continuity (in time) of energies due to the subdiffusive nature of the issue. Moreover, the abstract theory is applied to the Cauchy--Dirichlet problem for some $p$-Laplace subdiffusion equations with blow-up terms complying with the so-called Sobolev (sub)critical growth condition.
\end{abstract}

\section{Introduction}\label{S:Intro}

Let $H$ be a real Hilbert space equipped with an inner product $(\,\cdot\,,\,\cdot\,)_H$ and the norm $\|\cdot\|_H = \sqrt{(\,\cdot\,,\,\cdot\,)_H}$. For $i = 1,2$, let $\varphi^i : H \to (-\infty,\infty]$ be a proper (i.e., $\varphi^i \not\equiv \infty$) lower-semicontinuous convex functional with the \emph{effective domain},
$$
D(\varphi^i) := \left\{ w \in H \colon \varphi^i(w) < \infty \right\} \neq \emptyset,
$$
and define the \emph{subdifferential operator} $\partial \varphi^i : H \to 2^H$ of $\varphi^i$ by
$$
\partial \varphi^i(z) := \left\{ \xi \in H \colon \varphi^i(v)-\varphi^i(z) \geq (\xi,v-z)_H \ \mbox{ for } v \in D(\varphi^i)\right\}
$$
for $z \in D(\varphi^i)$ with the domain $D(\partial \varphi^i) = \{w \in D(\varphi^i) \colon \partial \varphi^i(w) \neq \emptyset\}$. In this article, we are concerned with the abstract Cauchy problem for time-fractional evolution equations of the form,
\begin{equation}\label{ee}
 \frac{\d}{\d t} \left[ k * (u - u_0) \right] (t) + \partial \varphi^1(u(t)) - \partial \varphi^2(u(t)) \ni f(t) \ \mbox{ in } H, \quad 0 < t < T,
\end{equation}
where $T \in (0,\infty]$, $u_0 \in D(\varphi^1) \cap D(\varphi^2)$ and $f : (0,T) \to H$ are prescribed, and the convolution $k * (u-u_0)$ with a kernel $k \in L^1_{\rm loc}([0,\infty))$ is defined by
$$
(k * w)(t) := \int^t_0 k(t-s) w(s) \, \d s \quad \mbox{ for } \ w \in L^1_{\rm loc}([0,\infty);H), \ \ t \geq 0.
$$
Here and henceforth, we assume:
\begin{enumerate}
 \item[({K})] The kernel $k \in L^1_{\rm loc}([0,\infty))$ is nonnegative and nonincreasing. There exists a nonnegative and nonincreasing kernel $\ell \in L^1_{\rm loc}([0,\infty))$ such that
$$
(k * \ell)(t) = \int^t_0 k(t-s) \ell(s) \, \d s = 1 \quad \mbox{ for } \ t \geq 0.
$$
\end{enumerate}
Hence $k$ is a \emph{completely positive kernel} (see~\cite[Theorem 2.2]{ClNo81}). A typical example of $k$ satisfying ({K}) is the so-called \emph{Riemann--Liouville kernel},
\begin{equation}\label{RL}
k_\alpha(t) = \frac{t^{-\alpha}}{\Gamma(1-\alpha)} \quad \mbox{ for } \ t > 0 \ \mbox{ and } \ 0 < \alpha < 1,
\end{equation}
and then, the nonlocal time-derivative $(\d/\d t)[k_\alpha * (u - u_0)]$ corresponds to the $\alpha$-th order Riemann--Liouville derivative $\partial_t^\alpha(u-u_0)$ of $u - u_0$, which also coincides with the $\alpha$-th order Caputo derivative of $u$ if it is smooth.

The study of classical gradient flows dates back to the early work of Ha\"im Br\'ezis (see, e.g.,~\cite{HB1}) concerning the evolution equation
$$
\dfrac{\d u}{\d t}(t) + \partial \varphi(u(t)) \ni f(t) \ \mbox{ in } H, \quad 0 < t < T, \quad u(0) = u_0
$$
for a proper lower-semicontinuous \emph{convex} functional $\varphi$ on a Hilbert space $H$. %Here the convexity can be relaxed to semiconvexity, i.e., the case where $\varphi + (\lambda/2)\|\cdot\|_H^2$ is convex for some $\lambda \in \R$. 
The work of Br\'ezis is nowadays known as \emph{Br\'ezis--K\=omura theory} (see also~\cite{Komura}) and has been extended in various directions\prf{ (see, e.g.,~\cite{Ken75} for time-dependent subdifferentials)}. Among those, \^Otani~\cite{Otani77} significantly relaxed the convexity assumption of the functional $\varphi$ by studying the Cauchy problem for evolution equations governed by the difference of two subdifferential operators such as
\begin{equation}\label{otani77}
\dfrac{\d u}{\d t}(t) + \partial \varphi^1(u(t)) - \partial \varphi^2(u(t)) \ni f(t) \ \mbox{ in } H, \quad 0 < t < T, \quad u(0) = u_0,
\end{equation}
which can be regarded as gradient flows for (essentially) \emph{nonconvex} energies (see also~\cite{Otani79} for long-time dynamics and~\cite{AS16} for a variational analysis). Actually, in contrast with gradient flows for \emph{semiconvex} energies, the abstract Cauchy problem \eqref{otani77} can cover (even degenerate or singular) parabolic equations whose solutions blow up in finite time, and hence, one cannot generally expect global (in time) existence of solutions to \eqref{otani77}. We further refer the reader to~\cite{Otani82,Otani84} for a nonmonotone perturbation theory for (convex) gradient flows. Extensions of these results on classical gradient flows to time-fractional variants have been pursued only for gradient flows of \emph{convex} energies. In~\cite{A19}, the well-posedness of the abstract Cauchy problem
\begin{equation*}
 \frac{\d}{\d t} \left[ k * (u - u_0) \right] (t) + \partial \varphi(u(t)) \ni f(t) \ \mbox{ in } H, \quad 0 < t < T
\end{equation*}
is proved for proper lower-semicontinuous (semi)convex functionals $\varphi:H \to (-\infty,\infty]$ and $f \in L^2(0,T;H)$, and moreover, a Lipschitz perturbation theory is developed. Furthermore, the abstract theory is also applied to time-fractional variants of \emph{nonlinear diffusion equations} as well as \emph{Allen-Cahn equations}. We also refer the reader to~\cite{mono} and references therein for related works (see also~\cite{AV19,AB23,GW20,KRY20,ScWi24,VeZa08,VeZa10,VeZa15,VeZa17,WWZ21,Za05,Za08,Za13,Za10,Za19}). On the other hand, time-fractional gradient flows for (essentially) \emph{nonconvex} energies such as \eqref{ee} have never been studied so far.

The main purpose of the present paper is to establish an abstract theory on existence of strong solutions (continuous on $[0,T)$) to the Cauchy problem \eqref{ee}. To this end, there arise several significant difficulties from the presence of time-fractional derivatives. We immediately face a difficulty from the defect of chain-rule formula for time-fractional derivatives. Indeed, the chain-rule formula is a crucial device to analyze gradient flow equations. There is an alternative formula established in~\cite{A19}, from which a fractional chain-rule \emph{inequality} for subdifferentials is derived in a practical form instead of usual identities; however, it is still insufficient to analyze the present issue \eqref{ee}. Moreover, the semigroup property does not hold true for the time-fractional evolution equations, that is, the concatenation of two solutions may not be a solution any more. As a result, one cannot construct a global solution to \eqref{ee} by concatenating some local solutions (see Remark \ref{R:semigr} below). We further realize another difficulty due to the lack of continuity (in time) of the energy $t \mapsto \varphi^1(u(t))$, which is actually from the subdiffusive nature of the issue. In the previous studies on non-fractional equations, various arguments based on the continuity are potentially used here and there, e.g., to construct local (in time) solutions to evolution equations (see, e.g.,~\cite{Ishii77}). Hence we may need to develop an alternative scheme to establish a local existence result without exploiting the continuity of the energy.

Another aim of this paper is to apply the preceding abstract theory to the following Cauchy--Dirichlet problem for \emph{$p$-Laplace subdiffusion equations with blow-up terms}:
\begin{alignat}{4}
\partial_t^\alpha (u-u_0) - \Delta_p u - |u|^{q-2}u &= f \ &&\mbox{ in } \ \Omega \times (0,\infty),\label{pde}\\
u &= 0 &&\mbox{ on } \partial \Omega \times (0,\infty),\label{bc}
\end{alignat}
where $1 < p,q < \infty$, $\Omega$ is a smooth bounded domain of $\R^d$ with boundary $\partial \Omega$, $\partial_t^\alpha (u-u_0) = \partial_t [k_\alpha * (u-u_0)]$ denotes the Riemann--Liouville derivative of $u-u_0$ (i.e., Caputo derivative of $u$) of order $0 < \alpha < 1$ (see \eqref{RL}), $\Delta_p$ is the so-called \emph{$p$-Laplacian} given as
$$
\Delta_p w = \nabla \cdot \left(|\nabla w|^{p-2}\nabla w\right)
$$
and $u_0 = u_0(x)$ and $f = f(x,t)$ are prescribed. Equation \eqref{pde} is a time-fractional variant of $p$-Laplace diffusion equations with blow-up terms of the form
\begin{equation}
 \partial_t u - \Delta_p u - |u|^{q-2}u = f  \ \mbox{ in } \ \Omega \times (0,\infty),\label{pde-c}
\end{equation}
along with the initial condition $u|_{t = 0} = u_0$ in $\Omega$ (cf.~see~\cite{Fujita}). Equation \eqref{pde-c} was proposed in~\cite{Lions} and then studied in~\cite{Tsutsumi}, where local and global (in time) existence of weak solutions to the Cauchy--Dirichlet problem for \eqref{pde-c} is proved for $u_0 \in W^{1,p}_0(\Omega)$ based on Galerkin's method. Although the global existence is proved under the so-called Sobolev subcritical condition $q < p^* := dp/(d-p)_+$, the local existence is verified under a somewhat restrictive assumption on $q$. The abstract results established in~\cite{Otani77,Ishii77,Otani79,Otani82} are also applied to \eqref{pde-c} for studying existence and asymptotic behavior of its $L^2$-solutions under the assumption that $q \leq p^*/2 + 1$, which is strictly less than $p^*$ (on the other hand, it can be relaxed to $q < 2^*$ for the semilinear case $p = 2$ in~\cite{Otani82,Otani84}). Moreover, in~\cite{AO05}, an abstract theory for \eqref{otani77} is developed in a reflexive Banach space setting and it is also used to prove local (in time) existence of weak solutions to the Cauchy--Dirichlet problem for \eqref{pde-c} under the subcritical condition $q < p^*$ for $u_0 \in W^{1,p}_0(\Omega)$. Furthermore, in~\cite{A07}, local existence of weak solutions to the Cauchy--Dirichlet problem for \eqref{pde-c} is eventually proved for any $u_0 \in L^r(\Omega)$ and $q > p \geq 2$, provided that $r > d(q-p)/p$; this result can cover all the local existence results mentioned above \prf{(within the frame of weak solutions) }and even be regarded as an extension of Weissler's results~\cite{Wei80} on the (semilinear) Fujita equation. We also refer the reader to~\cite{Nakao86,FuOh96}. On the other hand, to the best of the authors' knowledge, there have been no contribution to the study of the Cauchy--Dirichlet problem \eqref{pde}, \eqref{bc} for the $p$-Laplace subdiffusion equation so far. In this paper, thanks to the abstract theory on \eqref{ee} along with a nonlinear Calder\'on--Zygmund theory, we prove local (in time) existence of $L^2$-solutions to \eqref{pde}, \eqref{bc} under the Sobolev subcritical condition $q < p^*$ for any $2d/(d+2) < p < \infty$. Moreover, global (in time) existence of $L^2$-solutions for small data is also verified for $q \leq p^*$ even including the Sobolev critical exponent $p^*$.

\bigskip
\noindent
{\bf Plan of the paper.} This paper is composed of eight sections. In Section \ref{S:main}, we shall exhibit main results of the present paper, which consist of local and global (in time) existence of strong solutions to the abstract Cauchy problem \eqref{ee} under certain assumptions. We also stress that the strong solutions are continuous on $[0,T)$. Section \ref{S:pre} is concerned with preliminary material on subdifferential operators (see \S \ref{Ss:subdif}) as well as nonlocal time-differential operators (see \S \ref{Ss:fracderi}). In Section \ref{S:dev}, we shall develop time-fractional chain-rule formulae for subdifferentials (see \S \ref{Ss:ch}), continuous representatives of some convolutions (see \S \ref{Ss:conti-rep}) and a Lipschitz perturbation theory for convex gradient flows in Hilbert spaces (see \S \ref{Ss:Lip}) in order to prove the main results. Sections \ref{S:LE}--\ref{S:GE} are devoted to giving proofs of the main results: Section \ref{S:LE} gives a proof of Theorem \ref{T:LE} concerning the local (in time) existence result. In Section \ref{S:SDGE}, we shall prove Theorem \ref{T:SDGE} and Corollary \ref{C:SDGEc}, which are concerned with the global (in time) existence of strong solutions for small data. Moreover, another global existence result (see Theorem \ref{T:GE}) will be proved in Section \ref{S:GE}. Section \ref{S:app} is dedicated to discussing applications of the abstract theory to the Cauchy--Dirichlet problem for $p$-Laplace subdiffusion equations with blow-up terms. In Appendix, we shall give an outline of a proof for the solvability of some approximate problems which will be used only in \S \ref{S:GE}.

\bigskip
\noindent
{\bf Notation.} For each $1 < r < \infty$, we denote by $r'$ the H\"{o}lder conjugate of $r$, that is, $1/r+1/r'=1$. Moreover, we use the same letter $I$ for identity mappings defined on any spaces when no confusion can arise. We denote by $C$ a generic nonnegative constant which may vary from line to line.

\section{Main results}\label{S:main}

In this section, we present main results concerning existence of strong solutions to the abstract Cauchy problem \eqref{ee}. Throughout this paper, let $H$ be a real Hilbert space equipped with an inner product $(\cdot,\cdot)_H$ and the norm $\|\cdot\|_H = \sqrt{(\cdot,\cdot)_H}$, and moreover, we are concerned with strong solutions to the Cauchy problem \eqref{ee} in the following sense:

\begin{definition}[Strong solutions to \eqref{ee}]\label{D:sol}
Let $T \in (0,\infty]$. For any $S \in (0,T)$ (and $S \in (0,T]$ if $T < \infty$), a function $u \in L^2(0,S;H)$ is called a \emph{strong solution on} $[0,S]$ to the Cauchy problem \eqref{ee}, if the following conditions are all satisfied\/{\rm :}
\begin{enumerate}
 \item[\rm (i)] It holds that $k*(u-u_0) \in W^{1,2}(0,S;H)$ and $[k*(u-u_0)](0) = 0$. Moreover, $u(t) \in D(\partial \varphi^1) \cap D(\partial \varphi^2)$ for a.e.~$t \in (0,S)$.
 \item[\rm (ii)] There exist $\xi, \eta \in L^2(0,S;H)$ such that
\begin{equation}\label{EQ}
\left\{
\begin{aligned}
&\xi(t) \in \partial \varphi^1(u(t)), \quad \eta(t) \in \partial \varphi^2(u(t)),\\
&\dfrac{\d}{\d t} \left[ k * (u-u_0) \right] (t) + \xi(t) - \eta(t) = f(t)
\end{aligned}
\right.
\end{equation}
for a.e.~$t \in (0,S)$.
\end{enumerate}
Moreover, a function $u \in L^2_{\rm loc}([0,\infty);H)$ is called a \emph{strong solution on} $[0,\infty)$ to the Cauchy problem \eqref{ee} with $T = \infty$, if the following conditions are all satisfied\/{\rm :}
\begin{enumerate}
 \item[${\rm (i)}'$] It holds that $k*(u-u_0) \in W^{1,2}_{\rm loc}([0,\infty);H)$ and $[k*(u-u_0)](0) = 0$. Moreover, $u(t) \in D(\partial \varphi^1) \cap D(\partial \varphi^2)$ for a.e.~$t \in (0,\infty)$.
 \item[${\rm (ii)}'$] There exist $\xi, \eta \in L^2_{\rm loc}([0,\infty);H)$ such that \eqref{EQ} holds for a.e.~$t \in (0,\infty)$.
\end{enumerate}
\end{definition}

\begin{remark}[Lack of semigroup property]\label{R:semigr}
We emphasize that the semigroup property fails for strong solutions to the Cauchy problem \eqref{ee} due to the nonlocal nature of the equation, that is, the concatenation 
$$
u(t) := \begin{cases}
	u_1(t) &\mbox{ if } \ t \in [0,T),\\
	u_2(t-T) &\mbox{ if } \ t \in [T,2T]
       \end{cases}
$$
of two strong solutions $u_1, u_2$ on $[0,T]$ to \eqref{ee} does not always solve \eqref{ee} on $[0,2T]$ even if $u_1(T) = u_2(0)$. Hence, the lack of semigroup property prevents us to construct (global) strong solutions on $[0,\infty)$ by concatenation. %This is also why strong solutions to \eqref{ee} on $[0,\infty)$ are explicitly defined as in Definition \ref{D:sol}. 
Moreover, it is noteworthy that, in contrast with classical gradient flows, the existence of strong solutions to \eqref{ee} on $[0,\infty)$ does not follow immediately from the existence of those on $[0,T]$ for an arbitrary $T > 0$, unless solutions are uniquely determined by initial data.
\end{remark}

In order to state main results of the present paper, we set up the following assumptions:
\begin{enumerate}
 \item[(A1)] For each $r \in \R$, the sublevel set $\{w \in H \colon \varphi^1(w) + \|w\|_H \leq r\}$ is (pre)compact in $H$.
 \item[(A2)] It holds that $D(\partial \varphi^1) \subset D(\partial \varphi^2)$, and moreover, there exist a constant $\nu_1 \in [0,1)$ and a nondecreasing function $M_1 : \R \to [0,\infty)$ such that
$$
\|\mathring{\partial \varphi^2}(w)\|_H \leq \nu_1 \|\mathring{\partial \varphi^1}(w)\|_H + M_1(\varphi^1(w) + \|w\|_H)
$$
for $w \in D(\partial \varphi^1)$. Here $\mathring{\partial \varphi^i}$ denotes the minimal section of $\partial \varphi^i$ for $i = 1,2$, that is, $\mathring{\partial \varphi^i}(w)$ is the unique element of the set $\partial \varphi^i(w)$ such that $\|\mathring{\partial \varphi^i}(w)\|_H = \min \{ \|\xi\|_H \colon \xi \in \partial \varphi^i(w)\}$ for $w \in D(\partial \varphi^i)$.
\end{enumerate}

We are now ready to state main results. The following theorem is concerned with local existence of strong solutions to \eqref{ee}.

\begin{theorem}[Local existence of strong solutions to \eqref{ee}]\label{T:LE}
In addition to {\rm (K)}, assume that {\rm (A1)} and {\rm (A2)} hold and let $T \in (0,\infty)$. Then for every $u_0 \in D(\varphi^1)$ and $f \in L^2(0,T;H)$ satisfying $\ell * \|f(\cdot)\|_H^2 \in L^\infty(0,T)$, there exists $T_0 \in (0,T]$ such that the Cauchy problem \eqref{ee} admits a strong solution $u \in L^2(0,T_0;H)$ on $[0,T_0]$ such that
\begin{equation*}
u \in C([0,T_0];H), \quad u(0) = u_0 \quad \mbox{ and } \quad \varphi^1(u(\cdot)) \in L^\infty(0,T_0).
\end{equation*}
Here $T_0$ depends only on $\ell$, $\nu_1$, $M_1(\cdot)$, $\|u_0\|_H$, $\varphi^1(u_0)$ and $\| \ell * \|f(\cdot)\|_H^2 \|_{L^\infty(0,T)}$. 
\end{theorem}

\begin{remark}
Let $T \in (0,\infty)$. Since $\ell \in L^1(0,T)$, every $f \in L^\infty(0,T;H)$ fulfills $\ell * \|f(\cdot)\|_H^2 \in C([0,T])$. In particular, if $k = k_\alpha$ with $0 < \alpha < 1$, since $\ell = k_{1-\alpha}$ belongs to $L^{1/(1-\alpha),\infty}(0,T)$, one can check $\ell * \|f(\cdot)\|_H^2 \in C([0,T])$ for any $f \in L^{2/\beta}(0,T;H)$ with some $\beta \in (0,\alpha)$ due to Young's convolution inequality.
\end{remark}

We next present a theorem on global existence of strong solutions to \eqref{ee} for ``small'' initial data by assuming the following conditions instead of (A1) and (A2):
\begin{enumerate}
\item[${\rm (A1)}'$] For each $r \in \R$, the sublevel set $\{w \in H \colon \varphi^1(w) \leq r\}$ is (pre)compact in $H$.
\item[${\rm (A2)}'$] It holds that $D(\partial \varphi^1) \subset D(\partial \varphi^2)$ and $\varphi^1 \geq 0$, and moreover, there exist a constant ${\nu_2} \in [0,1)$, a nonnegative Borel measurable locally bounded function $m_2: [0,\infty) \to [0,\infty)$ and a nondecreasing function $M_2: [0,\infty) \to [0,\infty)$ satisfying 
\begin{equation}\label{M2-growth}
m_2(r) > 0 \ \mbox{ for } \ 0 < r < r_0 \quad \mbox{ and } \quad 
\lim_{r \to 0_+} \frac{M_{2}(r)}{m_2(r)} = 0
\end{equation}
for some $r_0 \in (0,\infty]$ such that
\begin{alignat}{4}
m_2(\varphi^1(w)) &\leq \|\mathring{\partial \varphi^1}(w)\|_{H},% \quad &&\mbox{ for all } \ w \in D(\partial \varphi^1),
\label{A2p-1}\\
\|\mathring{\partial \varphi^2}(w)\|_{H} &\leq {\nu_2} \|\mathring{\partial \varphi^1}(w)\|_{H} + M_{2}(\varphi^1(w))
% \quad &&\mbox{ for all } \ w \in D(\partial \varphi^1). 
\label{A2p-2}
\end{alignat}
for all $w \in D(\partial \varphi^1)$.
\end{enumerate}

\begin{remark}[Coercivity of $\varphi^1$ $\Rightarrow$ condition \eqref{A2p-1}]\label{R:A-coer}
If $\varphi^1$ fulfills the following coercivity condition:
\begin{equation}\label{A:coer}
\varphi^1(0) = 0 \quad \mbox{ and } \quad \alpha \|w\|_H^p \leq \varphi^1(w) \quad \mbox{ for } \ w \in D(\varphi^1)
\end{equation}
for some constants $\alpha,p > 0$ (hence $\varphi^1(w) > 0$ if and only if $w \neq 0$), then one has
\begin{align*}
\varphi^1(w) &\leq ( \mathring{\partial \varphi^1}(w), w )_H\\
&\leq \|\mathring{\partial \varphi^1}(w)\|_H \|w\|_H
\stackrel{\eqref{A:coer}}\leq \|\mathring{\partial \varphi^1}(w)\|_H \alpha^{-1/p} \{\varphi^1(w)\}^{1/p},
\end{align*}
which yields \eqref{A2p-1} with $m_2(s) := \alpha^{1/p} s^{1-1/p} > 0$ for $s > 0$ and $m_2(0):=0$. Moreover, ${\rm (A1)}'$ also follows from \eqref{A:coer} and (A1).
\end{remark}

\begin{theorem}[Global existence of strong solutions to \eqref{ee} for small data]\label{T:SDGE}
In addition to {\rm ({K})}, assume that ${\rm (A1)'}$ and ${\rm (A2)}'$ hold. Then there exists a constant $\delta_0 > 0$ depending only on $M_2(\cdot)$, $m_2(\cdot)$ and $\nu_2$ such that the following {\rm (i)} and {\rm (ii)} hold true\/{\rm :}
\begin{enumerate}
 \item[\rm (i)] In case $T \in (0,\infty)$, for every $u_0 \in D(\varphi^1)$ and $f \in L^2(0,T;H)$ satisfying 
$$
E_T(u_0,f) := \varphi^1(u_0) + \left\| \ell * \|f(\cdot)\|_H^2 \right\|_{L^\infty(0,T)} < \delta_0,
$$
the Cauchy problem \eqref{ee} admits a strong solution $u \in L^2(0,T;H)$ on $[0,T]$ satisfying
\begin{equation}\label{regu-T}
u \in  C([0,T];H), \quad u(0) = u_0 \quad \mbox{ and } \quad \varphi^1(u(\cdot)) \in L^\infty(0,T)
\end{equation} 
as well as
$$
\sup_{t \in [0,T]} \varphi^1(u(t)) \leq C E_T(u_0,f)
$$
for some constant $C > 0$ depending only on $\nu_2$.
 \item[\rm (ii)] In case $T = \infty$, for every $u_0 \in D(\varphi^1)$ and $f \in L^2_{\rm loc}([0,\infty);H)$ satisfying
\begin{equation}\label{E-inf-sm}
E_\infty(u_0,f) := \varphi^1(u_0) + \left\| \ell * \|f(\cdot)\|_H^2 \right\|_{L^\infty(0,\infty)} < \delta_0,
\end{equation}
the Cauchy problem \eqref{ee} with $T = \infty$ admits a strong solution $u \in L^2_{\rm loc}([0,\infty);H)$ on $[0,\infty)$ such that
\begin{equation*}
u \in  C([0,\infty);H), \quad u(0) = u_0 \quad \mbox{ and } \quad \varphi^1(u(\cdot)) \in L^\infty(0,\infty)
\end{equation*} 
and
$$
\sup_{t \in [0,\infty)} \varphi^1(u(t)) \leq C E_\infty(u_0,f)
$$
for some constant $C > 0$ depending only on $\nu_2$.
\end{enumerate}
\end{theorem}

\begin{remark}
When $\varphi^1(u_0) \ll 1$, a sufficient condition for $f \in L^2_{\rm loc}([0,\infty);H)$ satisfying \eqref{E-inf-sm} reads,
$$
\|f(t)\|_H^2 \leq \vep k(t) \quad \mbox{ for a.e. } t \in (0,\infty)
$$
for some constant $\vep > 0$ small enough.
\end{remark} 

We also give a corollary concerning global existence of strong solutions to \eqref{ee} for small data under the following critical growth condition  ${\rm (A2)}'_c$ instead of ${\rm (A2)}'$:
\begin{enumerate}
\item[${\rm (A2)}'_c$] It holds that $D(\partial \varphi^1) \subset D(\partial \varphi^2)$ and $\varphi^1 \geq 0$, and moreover, there exists a nondecreasing function $M_3: [0,\infty) \to [0,\infty)$ satisfying 
\begin{equation}\label{M2-growth-c}
M_3(r) > 0 \ \mbox{ for } \ r > 0 \quad \mbox{ and } \quad 
\lim_{r \to 0_+} M_{3}(r) = 0
\end{equation}
such that
\begin{alignat}{4}
\|\mathring{\partial \varphi^2}(w)\|_{H} &\leq \|\mathring{\partial \varphi^1}(w)\|_{H} M_{3}(\varphi^1(w))
\label{A2p-2c}
\end{alignat}
for all $w \in D(\partial \varphi^1)$.
\end{enumerate}

\begin{corollary}[Critical growth case]\label{C:SDGEc}
In addition to {\rm ({K})}, assume ${\rm (A1)'}$ and ${\rm (A2)}'_c$ above. Then there exists a constant $\delta_0 > 0$ depending only on $M_3(\cdot)$ such that {\rm (i)} and {\rm (ii)} of Theorem {\rm \ref{T:SDGE}} hold true with $C = 1$.
\end{corollary}

Finally, the following theorem asserts global existence of strong solutions to \eqref{ee} even for ``large'' initial data under an additional assumption,
\begin{enumerate}
 \item[(A3)] There exist constants ${\nu_3} \in [0,1)$, $r \in [0,2)$ and $c_1 \geq 0$ such that
$$
\varphi^2(w) \leq {\nu_3} \varphi^1(w) + c_1(\|w\|_H^r + 1) \quad \mbox{ for } \ w \in D(\varphi^1),
$$
which in particular yields $D(\varphi^1) \subset D(\varphi^2)$.
\end{enumerate}

\begin{theorem}[Global existence of strong solutions to \eqref{ee}]\label{T:GE}
In addition to {\rm ({K})}, assume that {\rm (A1), (A2)} and {\rm (A3)} hold. In case $T \in (0,\infty)$, for every $u_0 \in D(\varphi^1)$ and $f \in W^{1,2}(0,T;H)$, the Cauchy problem \eqref{ee} admits a strong solution $u \in L^2(0,T;H)$ on $[0,T]$ satisfying \eqref{regu-T}.

In case $T = \infty$, for every $u_0 \in D(\varphi^1)$ and $f \in W^{1,2}_{\rm loc}([0,\infty);H)$, then the Cauchy problem \eqref{ee} with $T = \infty$ admits a strong solution $u \in L^2_{\rm loc}([0,\infty);H)$ on $[0,\infty)$ satisfying \eqref{regu-T} for any $T > 0$.
\end{theorem}

\section{Preliminaries}\label{S:pre}

\subsection{Subdifferential operators}\label{Ss:subdif}

In this subsection, we give a brief summary on \emph{subdifferential operators} and \emph{maximal monotone operators}. We refer the reader to, e.g.,~\cite{HB1,B} for the details of the following material.

%Let $H$ be a real Hilbert space equipped with an inner product $(\cdot,\cdot)_H$ and the norm $\|\cdot\|_H = \sqrt{(\cdot,\cdot)_H}$. 

Let $\varphi : H \to (-\infty,\infty]$ be a proper (i.e., $\varphi \not\equiv \infty$) lower-semicontinuous convex functional (see, e.g.,~\cite{B-FA}). Set the \emph{effective domain} of $\varphi$ as
$$
D(\varphi) := \{w \in H \colon \varphi(w) < \infty\} \neq \emptyset.
$$
Then the \emph{subdifferential operator} $\partial \varphi : H \to 2^H$ of $\varphi$ is defined by
\begin{align*}
\partial \varphi(w) := \left\{ \xi \in H \colon \varphi(v) - \varphi(w) \geq (\xi, v-w)_H \ \mbox{ for all } \ v \in D(\varphi) \right\}
\end{align*}
for $w \in D(\varphi)$ with the domain 
$$
D(\partial \varphi) := \left\{ w \in D(\varphi) \colon \partial \varphi(w) \neq \emptyset \right\}.
$$
It is known that every subdifferential operator is maximal monotone in $H$. 

An operator $A : H \to 2^H$ is said to be \emph{maximal monotone} (or $m$-accretive) in the Hilbert space $H$, if the following two properties hold:
\begin{itemize}
 \item $(\xi_1-\xi_2,u_1-u_2)_H \geq 0$ \ for any $(u_1,\xi_1), (u_2,\xi_2) \in G(A)$,
 \item $R(I + \lambda A) = H$ \ for any (or some) $\lambda > 0$,
\end{itemize}
where $G(A)$ stands for the graph of the operator $A$ and $R(I+\lambda A)$ denotes the range of the operator $I + \lambda A$. We also denote by $\mathring{A}(w)$ the \emph{minimal section} of $A(w)$, that is, $\mathring{A}(w) = \mathrm{argmin} \{ \|\xi\|_H \colon \xi \in A(w)\}$, for $w \in D(A) := \{w \in H \colon A(w) \neq \emptyset\}$. Such a minimal section is uniquely determined for each $w \in D(A)$. Moreover, as every maximal monotone operator is \emph{demiclosedness}, we have the following useful property: Let $(u_n,\xi_n) \in G(A)$ be such that $u_n \to u$ strongly in $H$ and $\xi_n \to \xi$ weakly in $H$. Then $(u,\xi) \in G(A)$.

For each maximal monotone operator $A : H \to 2^H$, the \emph{resolvent} $J_\lambda : H \to H$ and \emph{Yosida approximation} $A_\lambda : H \to H$ are defined by
\begin{align*}
 J_\lambda = (I + \lambda A)^{-1}, \quad A_\lambda = \frac{I - J_\lambda}\lambda
\end{align*}
for $\lambda > 0$. The following properties are well known:
\begin{proposition}[Resolvents and Yosida approximations]\label{P:J-Y}
Let $A : D(A) \subset H \to H$ be a maximal monotone operator in a Hilbert space $H$ and let $J_\lambda : H \to H$ and $A_\lambda : H \to H$ be the resolvent and Yosida approximation of $A$, respectively, for $\lambda > 0$. Then the following {\rm (i)--(iv)} are all satisfied\/{\rm :}
\begin{enumerate}
 \item $J_\lambda : H \to H$ is non-expansive {\rm (}i.e., $1$-Lipschitz continuous{\rm )} in $H$ and $A_\lambda : H \to H$ is Lipschitz continuous in $H$ with the Lipschitz constant $1/\lambda$. Moreover, $A_\lambda$ is maximal monotone in $H$.
 \item $J_\lambda w \to w$ strongly in $H$ as $\lambda \to 0_+$ for every $w \in \overline{D(A)}^H$.
 \item $A_\lambda(w) \in A(J_\lambda w)$ for $w \in H$ and $\lambda > 0$.
 \item $\|A_\lambda(w)\|_H \leq \|\mathring{A}(w)\|_H$ for $w \in D(A)$ and $\lambda > 0$, and moreover, $A_\lambda(w) \to \mathring{A}(w)$ strongly in $H$ as $\lambda \to 0_+$ for $w \in D(A)$. 
\end{enumerate}
\end{proposition}

We next define the \emph{Moreau--Yosida regularization} $\varphi_\lambda : H \to \R$ of $\varphi$ for $\lambda > 0$ by
\begin{align*}
 \varphi_\lambda(w) &:= \inf_{z \in H} \left( \frac 1 {2\lambda} \|w - z\|_H^2 + \varphi(z) \right)\\
&\ = \frac 1 {2\lambda} \|w - J_\lambda w\|_H^2 + \varphi(J_\lambda w)
= \frac \lambda 2 \|A_\lambda(w)\|_H^2 + \varphi(J_\lambda w) \quad \mbox{ for } \ w \in H,
\end{align*}
where $J_\lambda$ and $A_\lambda$ stand for the resolvent and the Yosida approximation of $A := \partial \varphi$, respectively. Then we recall that
\begin{proposition}[Moreau--Yosida regularizations]\label{P:MY}
Let $\varphi : H \to (-\infty,\infty]$ be a proper lower-semicontinuous convex functional defined in a Hilbert space $H$ and let $\varphi_\lambda : H \to \R$ be the Moreau--Yosida regularization of $\varphi$ for $\lambda > 0$. Then the following {\rm (i)--(iii)} are all satisfied\/{\rm :}
\begin{enumerate}
 \item The Moreau--Yosida regularization $\varphi_\lambda: H \to \R$ is convex and Fr\'echet differentiable in $H$, and moreover, its derivative $\partial (\varphi_\lambda)$ coincides with the Yosida approximation $(\partial \varphi)_\lambda$ of $\partial \varphi$. Hence we shall simply write $\partial \varphi_\lambda$ instead of both $\partial (\varphi_\lambda)$ and $(\partial \varphi)_\lambda$ in what follows.
 \item $\varphi(J_\lambda w) \leq \varphi_\lambda(w) \leq \varphi(w)$ for $w \in H$ and $\lambda > 0$. Here $J_\lambda : H \to H$ stands for the resolvent of $\partial \varphi$.
 \item $\varphi_\lambda(w) \to \varphi(w)$ as $\lambda \to 0_+$ for $w \in H$.
\end{enumerate}
\end{proposition}

Define a functional $\Phi : L^2(0,T;H) \to (-\infty,\infty]$ by
$$
\Phi(w) := \begin{cases}
	    \int^T_0 \varphi(w(t)) \, \d t &\mbox{ if } \ \varphi(w(\cdot)) \in L^1(0,T),\\
	    \infty &\mbox{ otherwise}
	   \end{cases}
$$
for $w \in L^2(0,T;H)$. Then $\Phi$ is proper, lower-semicontinuous and convex in $L^2(0,T;H)$, and moreover, it is known that
$$
\xi \in \partial \Phi(w) \quad \mbox{ if and only if } \quad \xi(t) \in \partial \varphi(w(t)) \ \mbox{ for a.e. } t \in (0,T)
$$
for $w \in D(\Phi)$ and $\xi \in L^2(0,T;H)$. Furthermore, it also holds that
$$
\Phi_\lambda(w) = \int^T_0 \varphi_\lambda(w(t)) \, \d t, \quad \partial \Phi_\lambda(w) = \partial \varphi_\lambda(w(\cdot))
$$
for $w \in L^2(0,T;H)$ and $\lambda > 0$ (see also~\cite[IV.Example~2.C]{Showalter}).

Finally, we recall the following chain-rule formula for subdifferentials:
\begin{proposition}[See~{\cite[p.\,73, Lemma 3.3]{HB1}}]\label{P:std_chain}
Let $T \in (0,\infty)$ and let $u \in W^{1,2}(0,T;H)$ be such that $u(t) \in D(\partial \varphi)$ for a.e.~$t \in (0,T)$. Assume that there exists $g \in L^2(0,T;H)$ such that $g(t) \in \partial \varphi(u(t))$ for a.e.~$t \in (0,T)$. Then the function $t \mapsto \varphi(u(t))$ is absolutely continuous on $[0,T]$, and moreover, it holds that
$$
\left( \dfrac{\d u}{\d t}(t), h \right)_H = \dfrac{\d}{\d t} \varphi(u(t)) \quad \mbox{ for any } \ h \in \partial \varphi(u(t))
$$
for a.e.~$t \in (0,T)$.
\end{proposition}

\subsection{Nonlocal time-differential operators}\label{Ss:fracderi}

Let $T \in (0,\infty)$ and $p \in [1,\infty]$ be fixed and let $X$ be a Banach space. We denote by 
$$
\A : D(\A) \subset L^p(0,T;X) \to L^p(0,T;X)
$$
the (standard) \emph{time-differential operator} defined by
\begin{align*}
D(\A) &:= \left\{ w \in W^{1,p}(0,T;X) \colon w(0) = 0 \right\},\\
\A(w) &:= \dfrac{\d w}{\d t} \quad \mbox{ for } \ w \in D(\A).
\end{align*}
Then $\A$ is a linear $m$-accretive operator in $L^p(0,T;X)$. 

Let $k \in L^1_{\rm loc}([0,\infty))$ satisfy ({K}) (hence, $k$ is a completely positive kernel). Then the \emph{nonlocal time-differential operator} 
$$
\B : D(\B) \subset L^p(0,T;X) \to L^p(0,T;X)
$$ 
is defined by
\begin{align*}
 D(\B) &:= \left\{ w \in L^p(0,T;X) \colon k*w \in D(\A) \right\} \nonumber\\
&\; = \left\{ w \in L^p(0,T;X) \colon k*w \in W^{1,p}(0,T;X), \ (k*w)(0) = 0 \right\},\\
\B (w) &:= \A \left( k*w \right) = \dfrac{\d}{\d t} (k*w) \quad \mbox{ for } \ w \in D(\B).
\end{align*}
Then we observe that
$$
D(\A) \subset D(\B).
$$
%Indeed, for any $u \in D(\A)$, we find that $(k * u)(0) = 0$, since $k \in L^1(0,T)$ and $u \in W^{1,p}(0,T;X) \subset L^\infty(0,T;X)$. Moreover, we infer that $k * u \in W^{1,p}(0,T;X)$ by $u(0) = 0$ and $u \in W^{1,p}(0,T;X)$. Hence $u \in D(\B)$.
%
It is also known that $\B$ is a linear $m$-accretive operator in $L^p(0,T;X)$ under the assumption (K) (see, e.g.,~\cite{Cl84},~\cite{Gri85},~\cite{CP90},~\cite{VeZa08},~\cite{Za09}). Hence one can define the resolvent $\J_{1/n} : L^p(0,T;X) \to L^p(0,T;X)$ and the Yosida approximation $\B_{1/n} : L^p(0,T;X) \to L^p(0,T;X)$ of $\B$ for $n \in \N$ by
\begin{align*}
\J_{1/n} (w) &:= \left( I + n^{-1} \B \right)^{-1}(w), \\
\B_{1/n} (w) &:= n (w - \J_{1/n}(w)) = \B (\J_{1/n} w) = \dfrac{\d}{\d t} (k_n * w)
\end{align*}
for $w \in L^p(0,T;X)$. Here $k_n \in W^{1,1}(0,T)$ is a nonincreasing nonnegative kernel given by $k_n = n s_n$, where $s_n$ is a unique solution to the Volterra integral equation,
$$
s_n + n (\ell * s_n) = 1 \ \mbox{ in } (0,\infty).
$$
Then $k_n$ depends only on $\ell$ and $n$, and it is in particular independent of the choices of $X$ and $p$. Due to the general theory on linear $m$-accretive operators (see, e.g.,~\cite{Pazy},~\cite{B}), we assure that
\begin{equation}\label{Y-co}
\B_{1/n}(w) \to \B(w) \quad \mbox{ strongly in } L^p(0,T;X) \ \mbox{ as } \ n \to \infty,
\end{equation}
provided that $w \in D(\B)$. Indeed, $\B$ is densely defined in $L^p(0,T;X)$, and hence, for $w \in D(\B)$, we deduce that $\B_{1/n}(w) = \B(\J_{1/n}(w)) = \J_{1/n}(\B(w)) \to \B(w)$ strongly in $L^p(0,T;X)$ as $n \to \infty$. In particular, setting $X = \R$ and $p = 1$, we find that $w \equiv 1 \in D(\B)$, and hence, it follows from \eqref{Y-co} that
\begin{equation}\label{kn}
k_n \to k \quad \mbox{ strongly in } L^1(0,T) \ \mbox{ as } \ n \to \infty.
\end{equation}

In what follows, we always set
$$
X = H \quad \mbox{ and } \quad p = 2,
$$
unless noted otherwise, to consider the operators $\A$ and $\B$ for $k \in L^1_{\rm loc}([0,\infty))$ satisfying (K). Then $\A$ and $\B$ are maximal monotone in $L^2(0,T;H)$, and hence, thanks to Mazur's lemma along with the (demi)closedness of $\A$ and $\B$, the graphs $G(\A)$ and $G(\B)$ of $\A$ and $\B$, respectively, are weakly closed in $L^2(0,T;H) \times L^2(0,T;H)$.

Finally, we recall the following 
\begin{proposition}[See~{\cite[Proposition 3.5 and Corollary 3.6]{A19}}]\label{P:AB}
Let $T \in (0,\infty)$. Under the assumption {\rm (K)}, it holds that
    \begin{align*}
     \lefteqn{
      \int^t_0 \left( \A(u)(\tau), \B(u)(\tau) \right)_H \, \d \tau
      }\\
      &\geq \dfrac 1 2 \left( \ell * \|\B(u)(\cdot)\|_H^2 \right)(t) + \dfrac 1 2 \int^t_0 \ell(\tau) \|\B(u)(\tau)\|_H^2 \, \d \tau
    \end{align*}
    for all $u \in D(\A)$ and $t \in [0,T]$. It further implies the maximal monotonicity of $\A + \B$ in $L^2(0,T;H)$.
\end{proposition}

\section{Some devices}\label{S:dev}

This section is mainly devoted to developing \emph{integral forms of fractional chain-rule formulae} and a \emph{Lipschitz perturbation theory} for proving the main results stated in \S \ref{S:main}. Moreover, we also give a sufficient condition for the existence of continuous representatives of convolutions. %Throughout this section, let $H$ be a real Hilbert space equipped with an inner product $(\cdot,\cdot)_H$ and the norm $\|\cdot\|_H = \sqrt{(\cdot,\cdot)_H}$. 

\subsection{Fractional chain-rule formulae}\label{Ss:ch}

In this subsection, we develop some integral forms of (time-)fractional chain-rule formulae. To this end, we first recall the following nonlocal chain-rule formula for regular kernels, which is used in~\cite{A19} to derive differential forms of fractional chain-rule formulae:
\begin{proposition}[Nonlocal chain-rule formula~{\cite[Proposition 3.4]{A19}}]\label{P:nonloc_chain}
Let $T \in (0,\infty)$. Let $h \in W^{1,1}(0,T)$ be a nonnegative nonincreasing function and let $\varphi: H \to (-\infty,\infty]$ be a proper {\rm (}i.e., $\varphi \not\equiv \infty${\rm )} lower-semicontinuous convex functional. Let $u \in L^1(0,T;H)$ be such that $\varphi(u(\cdot)) \in L^1(0,T)$. Then for each $t \in (0,T)$ satisfying $u(t) \in D(\partial \varphi)$, it holds that
\begin{equation*}
\left(\dfrac{\d}{\d t} [h * (u-u_{0})](t), g \right)_{H}
\geq \dfrac{\d}{\d t} \left[ h * \big( \varphi(u(\cdot)) - \varphi(u_{0}) \big) \right](t)
\end{equation*}
for any $u_0 \in D(\varphi)$ and $g \in \partial \varphi(u(t))$.
\end{proposition}

In particular, if we choose $\varphi(\cdot) = (1/2)\|\cdot\|_H^2$, then we also have
\begin{equation*}
\left(\dfrac{\d}{\d t} [h * w](t), w(t) \right)_{H}
\geq \frac 12 \dfrac{\d}{\d t} \left[ h * \|w(\cdot)\|_H^2 \right](t)
\end{equation*}
for a.e.~$t \in (0,T)$ and $w \in L^2(0,T;H)$ (see~\cite{Za09}).

We next present some integral forms of fractional chain-rule formulae, which will play a crucial role to derive a priori estimates for approximate solutions to \eqref{ee}.

\begin{proposition}[Integral forms of fractional chain-rule formulae]\label{P:frac_chain}
Let $k \in L^1_{\rm loc}([0,\infty))$ be a function satisfying {\rm ({K})}. Let $\varphi: H \to (-\infty,\infty]$ be a proper {\rm (}i.e., $\varphi \not\equiv \infty${\rm )} lower-semicontinuous convex functional. Let $T \in (0,\infty)$, $u_0 \in D(\varphi)$ and let $u \in L^2(0,T;H)$ satisfy
\begin{itemize}
 \item $k*(u - u_0) \in W^{1,2}(0,T;H)$ and $[k*(u-u_0)](0) = 0$,
 \item $u(t) \in D(\partial \varphi)$ for a.e.~$t \in (0,T)$. 
\end{itemize}
Assume that there exists $g \in L^2(0,T;H)$ such that $g(t) \in \partial \varphi(u(t))$ for a.e.~$t \in (0,T)$. Then the following holds true\/{\rm :}
\begin{enumerate}
 \item[(i)] It holds that
$$
\int^t_0 \left(\dfrac{\d}{\d t} [k * (u-u_{0})](\tau), g(\tau) \right)_{H} \, \d \tau
\geq \left[ k * \big( \varphi(u(\cdot)) - \varphi(u_{0}) \big) \right](t)
$$
for a.e.~$t \in (0,T)$.
 \item[(ii)] Let $\ell$ be the conjugate kernel of $k$ {\rm (}see {\rm (K))}. It holds that
$$
\left[ \ell * \left(\dfrac{\d}{\d t} [k * (u-u_{0})](\cdot), g(\cdot) \right)_{H} \right](t)
\geq \varphi(u(t)) - \varphi(u_{0})
$$
for a.e.~$t \in (0,T)$.
\end{enumerate}
\end{proposition}

\begin{proof}
First of all, it follows from the assumption that $\varphi(u(\cdot)) \in L^1(0,T)$. Indeed, we see that
$$
\varphi(u(t)) \leq \varphi(u_0) + (g(t), u(t) - u_0)_H
$$
for a.e.~$t \in (0,T)$, and moreover, $u$ and $g$ belong to $L^2(0,T;H)$. Furthermore, $\varphi$ is affinely bounded from below (see, e.g.,~\cite{HB1}), we conclude that $\varphi(u(\cdot)) \in L^1(0,T)$.

Assertion (i) can be proved as in~\cite{A19}; however, we give a proof for the convenience of the reader. For $n \in \N$, let $k_n \in W^{1,1}(0,T)$ be a regularized kernel (see \S \ref{Ss:fracderi}). Thanks to Proposition \ref{P:nonloc_chain}, we have
\begin{equation}\label{ch-1}
\left(\dfrac{\d}{\d t} [k_n * (u-u_{0})](t), g(t) \right)_{H}
\geq \dfrac{\d}{\d t} \left[ k_n * \big( \varphi(u(\cdot)) - \varphi(u_{0}) \big) \right](t)
\end{equation}
for a.e.~$t \in (0,T)$. Integrate both sides over $(0,t)$, we infer that
$$
\int^t_0 \left(\dfrac{\d}{\d t} [k_n * (u-u_{0})](\tau), g(\tau) \right)_{H} \, \d \tau
\geq \left[ k_n * \big( \varphi(u(\cdot)) - \varphi(u_{0}) \big) \right](t)
$$
for $t \in [0,T]$. Since $\B_{1/n}(u-u_0) = (\d/\d t)[k_n * (u-u_0)] \to \B(u-u_0) = (\d/\d t)[k * (u-u_0)]$ strongly in $L^2(0,T;H)$ (by $u-u_0 \in D(\B)$ and \eqref{Y-co}) and $k_n \to k$ strongly in $L^1(0,T)$ (see \eqref{kn}), passing to the limit as $n \to \infty$, we deduce that
$$
\int^t_0 \left(\dfrac{\d}{\d t} [k * (u-u_{0})](\tau), g(\tau) \right)_{H} \, \d \tau \geq \left[ k * \big( \varphi(u(\cdot)) - \varphi(u_{0}) \big) \right](t)
$$
for a.e.~$t \in (0,T)$.

As for (ii), we convolve both sides of \eqref{ch-1} with $\ell$ to see that
\begin{align}
\lefteqn{
\left[ \ell * \left(\dfrac{\d}{\d t} [k_n * (u-u_{0})](\cdot), g(\cdot) \right)_{H} \right](t)
}\nonumber\\
&\geq \ell * \dfrac{\d}{\d t} \left[ k_n * \big( \varphi(u(\cdot)) - \varphi(u_{0}) \big) \right](t)\nonumber\\
&= \dfrac{\d}{\d t} \left[
\ell * k_n * \big( \varphi(u(\cdot)) - \varphi(u_{0}) \big)
\right](t)
\label{ch-2}
\end{align}
for a.e.~$t \in (0,T)$. Here we used the fact that $[ k_n * ( \varphi(u(\cdot)) - \varphi(u_{0}) ) ](0) = 0$. Moreover, we claim that
\begin{equation}\label{ch-3}
\ell * \left(\dfrac{\d}{\d t} [k_n * (u-u_{0})](\cdot), g(\cdot) \right)_{H}
\to \ell * \left(\dfrac{\d}{\d t} [k * (u-u_{0})](\cdot), g(\cdot) \right)_{H}
\end{equation}
strongly in $L^1(0,T)$ as $n \to \infty$. Indeed, we see that
\begin{align*}
\lefteqn{
\left\|
\ell * \left(\dfrac{\d}{\d t} [k_n * (u-u_{0})](\cdot), g(\cdot) \right)_{H}
- \ell * \left(\dfrac{\d}{\d t} [k * (u-u_{0})](\cdot), g(\cdot) \right)_{H}
\right\|_{L^1(0,T)}
}\\
&\leq \|\ell\|_{L^1(0,T)} \left\| \B_{1/n}(u-u_0) - \B(u-u_0) \right\|_{L^2(0,T;H)} \|g\|_{L^2(0,T;H)}
\to 0.
\end{align*}
We next test both sides of \eqref{ch-2} by a nonnegative test function $\phi \in C^\infty_c(0,T)$. It then follows that
\begin{align*}
\lefteqn{
\int^T_0 \phi(t) \left[ \ell * \left(\dfrac{\d}{\d t} [k_n * (u-u_{0})](\cdot), g(\cdot) \right)_{H} \right](t) \, \d t
}\nonumber\\
&\geq \int^T_0 \phi(t) \dfrac{\d}{\d t} \left[
\ell * k_n * \big( \varphi(u(\cdot)) - \varphi(u_{0}) \big)
\right](t) \, \d t\\
&= - \int^T_0 \phi'(t) \left[
\ell * k_n * \big( \varphi(u(\cdot)) - \varphi(u_{0}) \big)
\right](t) \, \d t\\
&\to - \int^T_0 \phi'(t) \left[
\ell * k * \big( \varphi(u(\cdot)) - \varphi(u_{0}) \big)
\right](t) \, \d t\\
&\hspace{-1mm}\stackrel{\text{(K)}}=  - \int^T_0 \phi'(t) \left( \int^t_0 \big[ \varphi(u(\tau)) - \varphi(u_{0}) \big] \, \d \tau \right) \, \d t\\
&=  \int^T_0 \phi(t) \big[ \varphi(u(t)) - \varphi(u_{0}) \big] \, \d t.
\end{align*}
Hence combining the above with \eqref{ch-3}, we conclude that
\begin{align*}
\lefteqn{
\int^T_0 \phi(t) \left[ \ell * \left(\dfrac{\d}{\d t} [k * (u-u_{0})](\cdot), g(\cdot) \right)_{H} \right](t) \, \d t
}\\
&\geq \int^T_0 \phi(t) \big[ \varphi(u(t)) - \varphi(u_{0}) \big] \, \d t,
\end{align*}
whence it follows from the arbitrariness of $\phi \geq 0$ that
$$
\left[ \ell * \left(\dfrac{\d}{\d t} [k * (u-u_{0})](\cdot), g(\cdot) \right)_{H} \right](t) \geq \varphi(u(t)) - \varphi(u_{0})
$$
for a.e.~$t \in (0,T)$. This completes the proof.
\end{proof}

\subsection{Continuous representatives of convolutions}\label{Ss:conti-rep}

In this subsection, we give some criterion on existence of continuous representatives of convolutions, which will be used to prove continuity of strong solutions to the Cauchy problem \eqref{ee} and to check initial conditions in a classical sense. Moreover, it will also be used in the next subsection concerned with the Lipschitz perturbation theory.

\begin{lemma}[Continuous representatives of convolutions]\label{L:conrep}
Let $T \in (0,\infty)$, $a \in H$ and let $g \in L^1(0,T)$ be nonnegative and nonincreasing. Let $G \in L^1(0,T;H)$ be such that $g * \|G(\cdot)\|_H^2 \in L^\infty(0,T)$ and define $u \in L^1(0,T;H)$ by
\begin{equation}\label{vol0}
u(t) = a + (g * G)(t) \quad \mbox{ for a.e.~} \ t \in (0,T).
\end{equation}
Then $u$ has a continuous representative $\tilde u \in C([0,T];H)$, that is, $u(t) = \tilde{u}(t)$ for a.e.~$t \in (0,T)$, such that
\begin{align*}
\|\tilde u(t)-\tilde u(s)\|_{H}
&\leq \|g\|_{L^1(0,|t-s|)}^{1/2} \left\| g * \|G(\cdot)\|_H^2 \right\|_{L^\infty(0,T)}^{1/2}\\
&\quad + \|g (|t-s| + \,\cdot\,) - g(\,\cdot\,)\|_{L^1(0,T-|t-s|)}^{1/2} \left\| g * \|G(\cdot)\|_H^2 \right\|_{L^\infty(0,T)}^{1/2}
\end{align*}
for all $t,s \in [0,T]$. Moreover, it holds that $\tilde u(t) \to a$ strongly in $H$ as $t \to 0_+$.
\end{lemma}

\begin{proof}
By assumption, there exists $I \subset (0,T)$ such that $|I| = 0$ and \eqref{vol0} holds for all $t \in (0,T) \setminus I$. Then letting $s,t \in (0,T) \setminus I$ satisfy $s \leq t$ and using H\"older's inequality, one observes that
\begin{align*}
\lefteqn{
\|u(t)-u(s)\|_{H}
}\\
&= \| (g * G)(t) - (g * G)(s) \|_{H}\\
&= \left\| \int^t_0 g(t-\tau) G(\tau) \, \d \tau - \int^s_0 g(s-\tau) G(\tau)\,\d \tau \right\|_{H}\\ 
&\leq \int^t_s g(t-\tau) \| G(\tau) \|_H \,\d \tau + \int^s_0 | g(t-\tau) - g(s-\tau) | \| G(\tau) \|_{H} \,\d \tau\\ 
&= \int^t_s g(t-\tau)^{1/2} g(t-\tau)^{1/2} \| G(\tau) \|_H \,\d \tau\\
&\quad + \int^s_0 | g(t-\tau) - g(s-\tau) |^{1/2} | g(t-\tau) - g(s-\tau) |^{1/2} \| G(\tau) \|_{H} \,\d \tau\\ 
&\leq \|g\|_{L^1(0,t-s)}^{1/2} \left\| g * \|G(\cdot)\|_H^2 \right\|_{L^\infty(0,T)}^{1/2}\\
&\quad + \|g (t-s + \,\cdot\,) - g(\,\cdot\,)\|_{L^1(0,T-(t-s))}^{1/2} \left\| g * \|G(\cdot)\|_H^2 \right\|_{L^\infty(0,T)}^{1/2}.
\end{align*}
Here we also used $| g(t-\tau) - g(s-\tau) | = g(s-\tau) - g(t-\tau) \leq g(s-\tau)$ from the nonincrease and nonnegativity of $g$ to derive the last inequality. For each $t \in [0,T]$ one can take a sequence $(t_n)$ in $(0,T) \setminus I$ converging to $t$, and then, $(u(t_n))$ forms a Cauchy sequence in $H$. Hence $\tilde u(t)$ can be defined as its limit, since the limit is also uniquely determined. Therefore there exists a continuous representative $\tilde u \in C([0,T];H)$ of $u$ such that $\tilde u$ fulfills the desired inequality for any $s,t \in [0,T]$. Moreover, repeating the same calculation above, we also find that
\begin{align*}
\|u(t)-a\|_{H} &\leq \|g\|_{L^1(0,t)}^{1/2} \left\| g * \|G(\cdot)\|_H^2 \right\|_{L^\infty(0,T)}^{1/2}
\end{align*}
for all $t \in (0,T) \setminus I$. Hence $\tilde u(t) \to a$ strongly in $H$ as $t \to 0_+$. This completes the proof.
\end{proof}

\subsection{Lipschitz perturbation theory}\label{Ss:Lip}

A Lipschitz perturbation theory to time-fractional gradient flows for convex energies is established in~\cite[\S 5]{A19}. In this subsection, it is further developed in order to derive better regularity (which can enable us to guarantee the initial condition in a classical sense) of strong solutions to the Cauchy problem,
\begin{equation}\label{Lip}
\frac{\d}{\d t} \left[ k * (u - u_0) \right] (t) + \partial \varphi(u(t)) + B(u(t)) \ni f(t) \ \mbox{ in } H, \quad 0 < t < T,
\end{equation}
where $T \in (0,\infty ]$, $\varphi : H \to (-\infty,\infty]$ is a proper lower-semicontinuous convex functional, $B : H \to H$ is a Lipschitz continuous operator, i.e., there exists a constant $L_B \geq 0$ such that
\begin{equation}\label{B}
\|B(w_1)-B(w_2)\|_H \leq L_B\|w_1-w_2\|_H \quad \mbox{ for }\ w_1,w_2 \in H,
\end{equation}
for every $u_0 \in D(\varphi)$ and $f \in L^2(0,T;H)$ satisfying $\ell * \|f(\cdot)\|_H^2 \in L^\infty(0,T)$. Strong solutions to \eqref{Lip} are defined analogously to Definition \ref{D:sol}. The following theorem will be used to construct approximate solutions to \eqref{ee}.

\begin{theorem}[Lipschitz perturbation theory (cf.~see~\cite{A19})]\label{T:Lip}
Assume that {\rm ({K})} and \eqref{B} are satisfied and let $T \in (0,\infty]$. In case $T \in (0,\infty)$, for every $u_0 \in D(\varphi^1)$ and $f \in L^2(0,T;H)$ satisfying $\ell * \|f(\cdot)\|_H^2 \in L^\infty(0,T)$, the Cauchy problem \eqref{Lip} admits a unique strong solution $u$ on $[0,T]$ such that
\begin{equation}\label{regu-Lip}
u \in C([0,T];H), \quad u(0) = u_0 \quad \mbox{ and } \quad \varphi(u(\cdot)) \in L^\infty(0,T).
\end{equation}
In case $T = \infty$, for every $u_0 \in D(\varphi^1)$ and $f \in L^2_{\rm loc}([0,\infty);H)$ satisfying $\ell * \|f(\cdot)\|_H^2 \in L^\infty_{\rm loc}([0,\infty))$, the Cauchy problem \eqref{Lip} with $T = \infty$ admits a unique strong solution $u$ on $[0,\infty)$ satisfying \eqref{regu-Lip} for every $T > 0$.
\end{theorem}

\begin{proof}
Let $T \in (0,\infty)$ be fixed. With the aid of Theorem 5.1 of~\cite{A19}, for every $u_0 \in D(\varphi)$ and $f \in L^2(0,T;H)$, one can assure that there exists a unique strong solution $u \in L^2(0,T;H)$ on $[0,T]$ to \eqref{Lip}, that is, there is $\xi \in L^2(0,T;H)$ such that
\begin{equation}\label{ee-Lip}
\dfrac{\d}{\d t} \left[ k * (u - u_0) \right](t) + \xi(t) + B(u(t)) = f(t), \quad \xi(t) \in \partial \varphi(u(t))
\end{equation}
for a.e.~$t \in (0,T)$. Hence it suffices to show \eqref{regu-Lip} under the (additional) assumption $\ell * \|f(\cdot)\|_H^2 \in L^\infty(0,T)$. We assume $\varphi \geq 0$ without loss of generality. Multiplying both sides of \eqref{ee-Lip} by $(\d/\d t) [k * (u-u_{0})](t)$, we have
\begin{align*}
\lefteqn{
\left\| \dfrac{\d}{\d t} \left[ k *(u - u_0) \right](t) \right\|_H^2 + \left(
\xi(t), \dfrac{\d}{\d t} \left[ k *(u - u_0) \right](t) \right)_H
}\\
&= - \left( B(u(t)), \dfrac{\d}{\d t} \left[ k *(u - u_0) \right](t) \right)_H + \left( f(t), \dfrac{\d}{\d t} \left[ k *(u - u_0) \right](t) \right)_H\\
&\leq \dfrac 12 \left\| \dfrac{\d}{\d t} \left[ k *(u - u_0) \right](t) \right\|_H^2 + \|B(u(t))\|_H^2 + \|f(t)\|_H^2\\
&\leq \dfrac 12 \left\| \dfrac{\d}{\d t} \left[ k *(u - u_0) \right](t) \right\|_H^2 + \left( \|B(u_0)\|_H + L_B\|u(t)-u_0\|_H \right)^2\\
&\quad + \|f(t)\|_H^2
\end{align*}
for a.e.~$t \in (0,T)$. Convolving both sides with $\ell$ and using (ii) of Proposition \ref{P:frac_chain}, we infer that
\begin{align}
\lefteqn{
\dfrac 12 \left( \ell * \left\| \dfrac{\d}{\d t} \left[ k * (u-u_{0}) \right] (\cdot) \right\|_{H}^{2} \right)(t) + \varphi(u(t)) - \varphi(u_{0})
}\nonumber\\
&\leq \left\| \ell * \|f(\cdot)\|_{H}^{2} \right\|_{L^\infty(0,T)}
+ 2 \|B(u_0)\|_{H}^{2} \|\ell\|_{L^1(0,T)} \nonumber\\
&\quad + 2 L_B^{2} \left( \ell * \|u(\cdot)-u_{0}\|_{H}^{2}\right)(t) \label{eL}
\end{align}
for a.e.~$t \in (0,T)$. Since $\varphi$ is nonnegative, one has
\begin{align*}
\lefteqn{
\left( \ell * \left\| \dfrac{\d}{\d t} \left[ k * (u-u_{0}) \right] (\cdot)\right\|_{H}^{2} \right)(t) + \varphi(u(t))
}\\
&\leq C\left[ \varphi(u_0) + \left\| \ell * \|f(\cdot)\|_{H}^{2} \right\|_{L^\infty(0,T)} + \|B(u_0)\|_{H}^{2} \|\ell\|_{L^1(0,T)} \right]\\
&\quad + C \left( \ell * \|u(\cdot)-u_{0}\|_{H}^{2} \right)(t) 
\end{align*}
for a.e.~$t \in (0,T)$. Here we note the following fact:
\begin{align*}
\|u(t)-u_{0}\|_{H}^{2}   
&= \left\| \dfrac{\d}{\d t} \left[ k * \ell * (u-u_{0}) \right](t) \right\|_{H}^{2}\\
&= \left\| \left( \ell * \dfrac{\d}{\d t} \left[ k*(u-u_{0}) \right] \right)(t) \right\|_{H}^{2}\\
&= \left\| 
\int_{0}^{t} \ell^{1/2}(t-s) \ell^{1/2}(t-s) \left( \dfrac{\d}{\d t} \left[ k*(u-u_{0}) \right] \right)(s) \, \d s
\right\|_{H}^{2}\\
&\leq
\left( \int_{0}^{t} \ell(t-s)  \, \d s \right)
\left( \int_{0}^{t} \ell(t-s) \left\| \dfrac{\d}{\d t} \left[k*(u-u_{0})\right](s) \right\|_H^{2} \, \d s \right)\\
&\leq \|\ell\|_{L^1(0,T)}
\left(
\ell * \left\| \dfrac{\d}{\d t} \left[k * (u-u_{0})\right](\cdot) \right\|_{H}^{2}
\right)(t)
\end{align*}
for~a.e.~$t \in (0,T)$. It follows that
\begin{align*}
\lefteqn{
\|u(t)-u_0\|_H^2
}\\
&\leq C \|\ell\|_{L^1(0,T)} \left[ \varphi(u_0) + \left\| \ell * \|f(\cdot)\|_{H}^{2} \right\|_{L^\infty(0,T)} + \|B(u_0)\|_{H}^{2} \|\ell\|_{L^1(0,T)} \right]\\
&\quad + C \|\ell\|_{L^1(0,T)} \left( \ell * \|u(\cdot)-u_{0}\|_{H}^{2} \right)(t) 
\end{align*}
for a.e.~$t \in (0,T)$. Thus applying a Gronwall-type lemma for a Volterra integral inequality (see Lemma \ref{Gron} below), we obtain
$$
\esssup_{t \in (0,T)} \|u(t)\|_H \leq C,
$$
which along with \eqref{eL} yields
$$
\esssup_{t \in (0,T)} \left[ \left( \ell * \left\| \dfrac{\d}{\d t} \left[ k * (u-u_{0}) \right] (\cdot)\right\|_{H}^{2} \right)(t) + \varphi(u(t)) \right] \leq C.
$$
Moreover, we recall that
$$
u(t) = u_0 + \left( \ell * \dfrac{\d}{\d t} \left[ k * (u-u_{0}) \right] \right)(t) \quad \mbox{ for a.e.~} \ t \in (0,T). 
$$
Therefore Lemma \ref{L:conrep} asserts that there exists a continuous representative of $u$ (still denoted by the same letter) and $u(0) = u_0$.

Finally, we prove existence of strong solutions on $[0,\infty)$ to the Cauchy problem \eqref{Lip} with $T = \infty$ for $u_0 \in D(\varphi^1)$ and $f \in L^2_{\rm loc}([0,\infty);H)$ satisfying $\ell * \|f(\cdot)\|_H^2 \in L^\infty_{\rm loc}([0,\infty))$. For each finite $T > 0$, from the fact which we have proved so far, the Cauchy problem \eqref{Lip} admits a unique strong solution $u^{(T)} \in C([0,T];H)$ on $[0,T]$ satisfying \eqref{regu-Lip}. Hence due to the uniqueness of strong solutions to \eqref{Lip} (i.e., $u^{(T)}|_{[0,S]} = u^{(S)}$ if $S \leq T$), one can define $u \in C([0,\infty);H)$ such that $u|_{[0,T]} = u^{(T)}$ for $T > 0$. We also set $\xi := f - B(u(\cdot)) - (\d/\d t)[k*(u-u_0)] \in L^2_{\rm loc}([0,\infty);H)$ and $\eta := B(u(\cdot)) \in C([0,\infty);H)$. Therefore it follows that $u \in C([0,\infty);H)$ is a strong solution on $[0,\infty)$ to \eqref{Lip} with $T = \infty$ and $u(0) = u_0$. This completes the proof.
\end{proof}

We close this subsection with the following Gronwall-type lemma for a Volterra integral inequality:

\begin{lemma}\label{Gron}
Let $S \in (0,\infty)$ and $1 \leq r \leq \infty$. Let $\phi \in L^1(0,S)$, $h \in L^{r}(0,S)$ and $g \in L^1(0,S)$ be nonnegative functions such that 
\begin{equation}\label{G-ineq}
\phi(t) \leq h(t) + (g * \phi)(t) \quad  \text{ for a.e.~} \ t \in (0,S).
\end{equation}
Then it holds that $\phi \in L^r(0,S)$, and moreover, there exists a constant $C_0 \geq 0$ depending only on $S$ and $g$ such that
$$
\|\phi \|_{L^{r}(0,S)} \leq C_0 \|h \|_{L^r(0,S)}.
$$
\end{lemma}

\begin{proof}
Let $(\rho_n)$ be a mollifier\footnote{The mollifier $(\rho_n)$ may be supposed to enjoy the following properties: $\rho_n \in C^\infty_c([0,\infty))$, $\rho_n \geq 0$, $\int^\infty_0 \rho_n(s) \, \d s = 1$ and $\mathrm{supp}\, \rho_n = [0,1/n]$ for $n \in \N$. Then for $1 \leq p < \infty$ and $T \in (0,\infty)$, one can check in a standard way that $\rho_n * w \in L^\infty(0,T)$ and $\rho_n * w \to w$ strongly in $L^p(0,T)$ for any $w \in L^p(0,T)$.} and convolve both sides of \eqref{G-ineq} with $\rho_n \geq 0$. It then follows that
\begin{equation}\label{Gro-n}
\phi_n(t) \leq h_n(t) + (g* \phi_n)(t) \quad  \text{ for a.e.~} \ t \in (0,S),
\end{equation}
where $\phi_n := \rho_n * \phi \in L^\infty(0,S)$ and $h_n := \rho_n * h \in L^\infty(0,S)$. We observe that
\begin{align*}
(g * \phi_n)(t)
&= \int^t_0 g(t-s) \phi_n(s) \, \d s\\
&= \e^{Mt} \int^t_0 g(t-s) \e^{-M(t-s)} \phi_n(s) \e^{-M s} \, \d s\\
&= \e^{Mt} \left[ (g \,\e^{-M{\,\cdot\,}}) * (\phi_n \,\e^{-M{\,\cdot\,}}) \right](t) \quad \mbox{ for } \ t \in [0,S],
\end{align*}
where $M$ is a constant which will be determined later and $\e^{-M{\,\cdot\,}}$ stands for the function $t \mapsto \e^{-M t}$ from $[0,\infty) \to \R$, and which along with the Young's convolution inequality yields
\begin{align*}
\|(g * \phi_n) \, \e^{-M{\,\cdot\,}} \|_{L^{r}(0,S)}
\leq \|g \,\e^{-M{\,\cdot\,}}\|_{L^1(0,S)} \|\phi_n \,\e^{-M{\,\cdot\,}}\|_{L^r(0,S)}.
\end{align*}
Hence it follows from \eqref{Gro-n} that
\begin{align*}
\lefteqn{
\|\phi_n \,\e^{-M{\,\cdot\,}}\|_{L^{r}(0,S)}
}\\
&\leq \|h_n \,\e^{-M{\,\cdot\,}}\|_{L^{r}(0,S)} + \|(g * \phi_n) \, \e^{-M{\,\cdot\,}}\|_{L^{r}(0,S)}\\
&\leq \|h_n \,\e^{-M{\,\cdot\,}}\|_{L^{r}(0,S)} + \|g \,\e^{-M{\,\cdot\,}}\|_{L^1(0,S)} \|\phi_n \,\e^{-M{\,\cdot\,}}\|_{L^r(0,S)}.
\end{align*}
We take a positive constant $M$ large enough so that
$$
\|g \,\e^{-M{\,\cdot\,}}\|_{L^1(0,S)} \leq \frac 12.
$$
Here we note that $M$ depends only on $S$ and $g$. Therefore we deduce that
$$
\|\phi_n\,\e^{-M{\,\cdot\,}}\|_{L^{r}(0,S)} \leq 2 \|h_n\,\e^{-M{\,\cdot\,}}\|_{L^{r}(0,S)}
$$
for $n \in \N$. We pass to the limit as $n \to \infty$. Note that $\phi_n \to \phi$ strongly in $L^1(0,S)$ and $h_n \to h$ strongly in $L^r(0,S)$. Hence as for $r > 1$, one finds that $\phi_n \to \phi$ weakly in $L^r(0,S)$ up to a subsequence. Thus for $r \in [1,\infty)$ we can reach 
$$
\|\phi \, \e^{-M{\,\cdot\,}}\|_{L^{r}(0,S)} \leq  2 \|h \, \e^{-M{\,\cdot\,}}\|_{L^r(0,S)},
$$
which yields the desired conclusion with the choice $C_0 = 2 \e^{MS}$.
\end{proof}

\section{Proof of Theorem \ref{T:LE}}\label{S:LE}

In this section, we give a proof for Theorem \ref{T:LE}. To this end, we first set up a Gronwall-type lemma for nonlinear Volterra integral inequalities.

\begin{lemma}\label{L:Gron1}
Let $S \in (0,\infty)$ and let $\phi \in L^{\infty}(0,S)$ be a nonnegative function. Suppose that there exist a constant ${a} \geq 0$, a nonnegative function $g \in L^1_{\rm loc}([0,\infty))$ and a nondecreasing function $M : [0,\infty) \to [0,\infty)$ such that
$$
\phi(t) \leq {a} + [g * M(\phi(\cdot))](t)  \quad \text{ for a.e.~} \ t \in (0,S).
$$
Then there exists a constant $R > 0$ which depends only on $M(a+1)$ and $g$ such that
$$
\|\phi\|_{L^{\infty}(0, \min \{R,S\})} \leq {a}+1.
$$
\end{lemma}

Here it is noteworthy that $\phi$ is not supposed to be continuous in Lemma \ref{L:Gron1}. Actually, in order to prove Theorem \ref{T:LE}, we shall apply Lemma \ref{L:Gron1} for some possibly \emph{discontinuous} functions $\phi \in L^\infty(0,T)$ due to the subdiffusive nature of the problem.

\begin{proof}
We can assume $M > 0$ in $(0,\infty)$ without loss of generality, and we denote by the same letter the zero extension of $\phi$ onto $[0,\infty)$. Since $g \in L^1_{\rm loc}([0,\infty))$, there exists a constant $R > 0$ such that 
\begin{align*}
\int_0^R g(t) \, \d t < \dfrac{1}{4M({a}+1)}.
\end{align*}
Here $R$ depends only on $M({a}+1)$ and $g$. Moreover, we note that
\begin{align*}
\phi(t) &\leq {a} + [g * M(\phi(\cdot))](t)\\
&\leq {a} + [g * M(\|\phi\|_{L^\infty(0,S)})](t) = {a} + \|g\|_{L^1(0,t)} M(\|\phi\|_{L^\infty(0,S)})
\end{align*}
for a.e.~$t \in (0,\infty)$. Hence we similarly observe that $\|\phi\|_{L^\infty(0,t)} \leq {a} + 1$ for $t > 0$ small enough. Set 
$$
T_{\phi} := \sup\{ t \in (0,\infty) \colon \|\phi\|_{L^{\infty}(0,t)} \leq {a} + 1\} > 0.
$$
We claim that $R \leq T_{\phi}$. Suppose to the contrary that $T_{\phi} < R$. Then since $g \in L^1_{\rm loc}([0,\infty))$, there exists $t_{\phi} \in (T_{\phi}, R)$ such that  
\begin{equation}\label{**}
\left\{
\begin{aligned}
&{a}+1 < \phi(t_\phi) \leq {a} + [g * M(\phi(\cdot))](t_\phi),\\
&M( \|\phi\|_{L^{\infty}(0,S)} ) \int_{0}^{t_{\phi} - T_{\phi}} g(t) \, \d t < \frac{1}{4}.
\end{aligned}
\right.
\end{equation}
\prf{Indeed, let $n \in \N$ be large enough so that $T_n := T_\phi + 1/n \in (T_\phi,R)$. Then we find that
$$
a + 1 < \|\phi\|_{L^\infty(0,T_n)} = \inf \{C \colon \phi(t) \leq C \ \mbox{ for a.e.~} \ t \in (0,T_n) \},
$$
which along with the fact that $\|\phi\|_{L^\infty(0,T_\phi)} \leq a + 1$ yields
$$
\phi(t) > a + 1 \quad \mbox{ for all } \ t \in I_n
$$
for some measurable set $I_n \subset (T_\phi,T_\phi + 1/n)$ having a positive measure $|I_n| > 0$. Hence choosing $n \in \N$ large enough, one can take $t_\phi \in I_n$ satisfying \eqref{**}.} 
Therefore we derive from the monotonicity of $M$ and the choices of $t_{\phi}$ and $T_{\phi}$ that
\begin{align*}
{a}+1 &< \phi(t_{\phi}) \leq {a} + [g * M(\phi(\cdot))](t_{\phi}) \\
&= {a} + \int_{0}^{T_{\phi}} g(t_{\phi} -t)  M(\phi(t)) \, \d t 
+ \int_{T_{\phi}}^{t_{\phi}} g(t_{\phi} -t) M(\phi(t)) \, \d t\\
&\leq {a} + M({a}+1) \int_{0} ^{T_{\phi}} g(t_{\phi} -t) \, 	\d t + M \left( \|\phi\|_{L^{\infty}(0,S)} \right) \int_{T_{\phi}}^{t_{\phi}} g(t_{\phi} -t) \, \d t\\
&= {a} + M({a}+1) \int^{t_\phi}_{t_\phi-T_{\phi}} g(\tau) \, \d \tau + M \left( \|\phi\|_{L^{\infty}(0,S)} \right) \int^{t_{\phi}-T_{\phi}}_0 g(\tau) \, \d \tau\\
&\leq {a} + M({a}+1) \int^R_0 g(\tau) \, \d \tau + M \left( \|\phi\|_{L^{\infty}(0,S)} \right) \int^{t_{\phi}-T_{\phi}}_0 g(\tau) \, \d \tau\\
&\leq {a} + M({a}+1) \dfrac{1}{4M({a}+1)} + \dfrac{1}{4} = {a} + \dfrac{1}{2},
\end{align*}
which yields a contradiction. Thus we conclude that $R \leq T_{\phi}$, which implies the desired conclusion.
\end{proof}

We are now ready to prove Theorem \ref{T:LE}.

\begin{proof}[Proof of Theorem {\rm \ref{T:LE}}]
Thanks to the affine boundedness of proper lower-semicontinuous convex functionals (see, e.g.,~\cite{HB1}), one can assume $\varphi^i \geq 0$ for $i = 1,2$ without loss of generality.\footnote{Indeed, for each $i = 1,2$, there exist $a_i \in \R$ and $z_i \in H$ such that $\varphi^i(w) \geq (z_i,w)_H + a_i$ for $w \in H$. Set $\hat \varphi^i(w) := \varphi^i(w) - (z_i,w)_H - a_i \geq 0$. We note that $\partial \hat \varphi^i(w) = \partial \varphi^i(w) - z_i$. Hence \eqref{ee} is rewritten as $(\d/\d t)[k * (u-u_0)](t) + \partial \hat \varphi^1(u(t)) - \partial \hat\varphi^2(u(t)) \ni f(t) - z_1 + z_2 =: \hat f(t)$. Moreover, $\hat\varphi^1$ and $\hat\varphi^2$ fulfill (A1) and (A2) as well.} Let $T \in (0,\infty)$ be fixed and consider the following approximate problems:
\begin{equation}\label{ee-aprx}
\frac{\d}{\d t} \left[ k * (u_\lambda - u_0) \right] (t) + \partial \varphi^1(u_\lambda(t)) - \partial \varphi^2_\lambda(u_\lambda(t)) \ni f(t) \ \mbox{ in } H, \quad 0 < t < T,
\end{equation}
where $\varphi^2_\lambda$ stands for the Moreau--Yosida regularization of $\varphi^2$, for $\lambda \in (0,1)$ (see \S \ref{Ss:subdif}). Since $\partial \varphi^2_\lambda : H \to H$ is Lipschitz continuous, due to Theorem \ref{T:Lip}, for $\lambda \in (0,1)$, $u_0 \in D(\varphi^1)$ and $f \in L^2(0,T;H)$ satisfying $\ell * \|f(\cdot)\|_H^2 \in L^\infty(0,T)$, we can prove that the Cauchy problem \eqref{ee-aprx} admits a strong solution $u_\lambda \in L^2(0,T;H)$ on $[0,T]$ such that 
$$
u_\lambda \in C([0,T];H), \quad u_\lambda(0) = u_0,\quad \varphi^1(u_\lambda(\cdot)) \in L^\infty(0,T).
$$
Let $g_\lambda \in L^2(0,T;H)$ be such that $g_\lambda(t) \in \partial \varphi^1(u_\lambda(t))$ and
\begin{equation}\label{ee-lam}
\frac{\d}{\d t} \left[ k * (u_\lambda - u_0) \right] (t) + g_\lambda(t) - \partial \varphi^2_\lambda(u_\lambda(t)) = f(t)
\end{equation}
for a.e.~$t \in (0,T)$.

We next establish a priori estimates uniformly for $\lambda \in (0,1)$. Test \eqref{ee-lam} by $g_\lambda(t)$ to see that
\begin{align*}
\lefteqn{
\left( \frac{\d}{\d t} \left[ k * (u_\lambda - u_0) \right] (t), g_\lambda(t) \right)_H + \|g_\lambda(t)\|_H^2
}\\
&= \left( \partial \varphi^2_\lambda(u_\lambda(t)), g_\lambda(t) \right)_H + \left(f(t), g_\lambda(t) \right)_H\\
&\leq \| \partial \varphi^2_\lambda(u_\lambda(t)) \|_H \| g_\lambda(t) \|_H + \| f(t) \|_H \|g_\lambda(t)\|_H
\end{align*}
for a.e.~$t \in (0,T)$. From (iv) of Proposition \ref{P:J-Y} and (A2), using Young's inequality, we find that 
\begin{align}
\lefteqn{
\| \partial \varphi^2_\lambda (u_\lambda(t)) \|_H
\leq \| \mathring{\partial \varphi^2}(u_\lambda(t))\|_H
}\nonumber\\
&\leq \nu_1 \|g_\lambda(t)\|_H + M_1( \varphi^1 (u_\lambda(t)) + \|u_\lambda(t)\|_H) \nonumber\\
&\leq \nu_1 \|g_\lambda(t)\|_H \nonumber \\
&\quad + M_1\left( (1/2)\|u_\lambda(t)-u_0 \|_H^2 + \varphi^1 (u_\lambda(t)) + \|u_0\|_H + 1 \right).\label{dphi2}
\end{align}
Hence we have
\begin{align*}
\lefteqn{
\left( \frac{\d}{\d t} \left[ k * (u_\lambda - u_0) \right] (t), g_\lambda(t) \right)_H + \|g_\lambda(t)\|_H^2
}\\
&\leq \nu_1 \|g_\lambda(t)\|_H^2 \\
&\quad + M_1 \left( (1/2)\|u_\lambda(t)-u_0 \|_H^2 + \varphi^1 (u_\lambda(t)) + \|u_0\|_H + 1 \right) \| g_\lambda(t) \|_H \\
&\quad + \| f(t) \|_H \| g_\lambda(t) \|_H
\end{align*}
for a.e.~$t \in (0,T)$. Therefore employing Young's inequality, one can take a constant $C_1 = C_1(\nu_1) > 0$ such that
\begin{align*}
\lefteqn{
\left( \frac{\d}{\d t} \left[ k * (u_\lambda - u_0) \right] (t), g_\lambda(t) \right)_H + \frac{1-\nu_1}2 \|g_\lambda(t)\|_H^2
}\\
&\leq C_1 M_1 \left( (1/2)\|u_\lambda(t)-u_0 \|_H^2 + \varphi^1 (u_\lambda(t)) + \|u_0\|_H + 1 \right)^2 + C_1 \| f(t) \|_H ^2
\end{align*}
for a.e.~$t \in (0,T)$. Convolving both sides with $\ell$ and employing (ii) of Proposition \ref{P:frac_chain}, we have 
\begin{align}
\lefteqn{
\varphi^1(u_\lambda(t)) - \varphi^1(u_0) + \frac{1-\nu_1}2 \left( \ell * \|g_\lambda(\cdot)\|_H^2 \right)(t)
} \nonumber\\
&\leq C_1 \left[ \ell * M_1 \left( (1/2)\|u_\lambda(\cdot)-u_0 \|_H^2 + \varphi^1(u_\lambda(\cdot)) + \|u_0\|_H + 1 \right)^2 \right](t) \nonumber\\
&\quad + C_1 \left(\ell * \| f(\cdot) \|_H^2 \right)(t)\label{ei1}
\end{align}
for a.e.~$t \in (0,T)$. 

Moreover, testing \eqref{ee-lam} by $u_\lambda(t)-u_0$, we derive from \eqref{dphi2} that
\begin{align*}
\lefteqn{
\left( \frac{\d}{\d t} \left[ k * (u_\lambda - u_0) \right] (t), u_\lambda(t) - u_0 \right)_H + \left(g_\lambda(t), u_\lambda(t)-u_0 \right)_H
}\\
&\leq \| \partial \varphi^2_\lambda(u_\lambda(t)) \|_H \| u_\lambda(t) - u_0 \|_H + \| f(t) \|_H \| u_\lambda(t) - u_0 \|_H
\end{align*}
for a.e.~$t \in (0,T)$. Recalling the definition of subdifferential and employing Young's inequality again, one can take a constant $C_2 = C_2(\nu_1) > 0$ such that
\begin{align*}
\lefteqn{
\left( \frac{\d}{\d t} \left[ k * (u_\lambda - u_0) \right] (t), u_\lambda(t) - u_0 \right)_H +  \varphi^1(u_\lambda(t)) - \varphi^1(u_0) 
}\\
&\leq \frac{1-\nu_1}4 \|g_\lambda(t)\|_H^2\\
&\quad + C_2 \Big\{ M_1 \left( (1/2)\|u_\lambda(t)-u_0 \|_H^2 + \varphi^1 (u_\lambda(t)) + \|u_0\|_H + 1 \right)^2 \\
&\quad + \|f(t)\|_H^2 + (1/2)\|u_\lambda(t)-u_0\|_H^2 \Big\}
\end{align*}
for a.e.~$t \in (0,T)$. Convolving both sides with $\ell$ and applying (ii) of Proposition \ref{P:frac_chain} to $\varphi(\cdot) := (1/2)\|\,\cdot\, - u_0\|_H^2$, we deduce from the nonnegativity of $\varphi^1$ that 
\begin{align}
\lefteqn{
\frac 12 \|u_\lambda(t) - u_0\|_H^2 
} \nonumber\\
&\leq \|\ell\|_{L^1(0,T)} \varphi^1(u_0) + \frac{1-\nu_1}4 \left( \ell * \|g_\lambda(\cdot)\|_H^2 \right)(t) \nonumber\\
&\quad + C_2 \left[ \ell * M_1 \left( (1/2)\|u_\lambda(\cdot)-u_0 \|_H^2 + \varphi^1 (u_\lambda(\cdot)) + \|u_0\|_H + 1 \right)^2 \right](t) \nonumber\\
&\quad + C_2 \left( \ell * \|f(\cdot)\|_H^2 \right)(t) + C_2 \left[ \ell * (1/2)\|u_\lambda(\cdot)-u_0\|_H^2 \right](t)\label{ei2}
\end{align}
for a.e.~$t \in (0,T)$. Combining \eqref{ei1} and \eqref{ei2}, we obtain
\begin{align}
\lefteqn{
\underbrace{
\frac 12 \|u_\lambda(t) - u_0\|_H^2 + \varphi^1(u_\lambda(t))
}_{=: \, \phi(t)} + \frac{1-\nu_1}4 \left( \ell * \|g_\lambda(\cdot)\|_H^2 \right)(t)
}\nonumber\\
&\leq a + \Big[ \ell * M \big( \underbrace{(1/2) \|u_\lambda(t) - u_0\|_H^2 + \varphi^1(u_\lambda(\cdot))}_{= \, \phi(\cdot)} \big) \Big](t)\label{ei3}
\end{align}
for a.e.~$t \in (0,T)$. Here $a$ and $M$ are given as
\begin{align*}
a &:= \left( \|\ell\|_{L^1(0,T)} + 1 \right) \varphi^1(u_0) + (C_1 + C_2) \left\| \ell * \| f(\cdot) \|_H^2 \right\|_{L^\infty(0,T)},\\
M(s) &:= C_2 s + (C_1 + C_2) M_1\left(s + \|u_0\|_H + 1\right)^2 \quad \mbox{ for }\ s \in \R.
\end{align*}
Hence applying Lemma \ref{L:Gron1}, we can take a constant $T_0 \in (0,T]$ such that
$$
\esssup_{t \in (0,T_0)} \left[ (1/2)\|u_\lambda(t) - u_0\|_H^2 + \varphi^1(u_\lambda(t)) \right] \leq C.
$$
Here and henceforth, $C$ denotes a generic constant being independent of $t$ and $\lambda$. We also note that $T_0$ depends only on $\ell$ and $M(a+1)$. Moreover, since $u_\lambda$ belongs to $C([0,T];H)$ and $\varphi^1$ is lower-semicontinuous in $H$, one can further verify that
\begin{equation}\label{est1}
\sup_{t \in [0,T_0]} \left[ (1/2)\|u_\lambda(t) - u_0\|_H^2 + \varphi^1(u_\lambda(t)) \right] \leq C.
\end{equation}
Recalling \eqref{ei3}, we also get
$$
\esssup_{t \in (0,T_0)} \left( \ell * \|g_\lambda(\cdot)\|_H^2 \right)(t) \leq C,
$$
which along with (K) yields
\begin{equation*}
\int^{T_0}_0 \|g_\lambda(t)\|_H^2 \, \d t \leq C.
\end{equation*}
Moreover, we deduce from (A2) and \eqref{est1} that
\begin{equation*}
\esssup_{t \in (0,T_0)} \left( \ell * \left\|\partial \varphi^2_\lambda (u_\lambda(\cdot)) \right\|_H^2 \right)(t) \leq C,
\end{equation*}
which leads us to obtain
\begin{equation}\label{est3}
\int^{T_0}_0 \left\|\partial \varphi^2_\lambda (u_\lambda(t)) \right\|_H^2 \, \d t \leq C.
\end{equation}
It further follows from \eqref{ee-aprx} that
\begin{equation}\label{est4-}
\esssup_{t \in (0,T_0)} \left( \ell * \left\| \dfrac{\d}{\d t} [k * (u_\lambda-u_0)] (\cdot) \right\|_H^2 \right)(t) \leq C,
\end{equation}
which also gives
\begin{equation*}
\int^{T_0}_0 \left\| \dfrac{\d}{\d t} [k * (u_\lambda-u_0)] (t) \right\|_H^2 \, \d t \leq C.
\end{equation*}

Furthermore, noting that
\begin{equation}\label{l*d-alpha}
\left( \ell * \dfrac{\d}{\d t} [k * (u_\lambda - u_0) ] \right)(t)
= u_\lambda(t) - u_{0} \quad \mbox{ for } \ t \in (0,T),
\end{equation}
we deduce from Lemma \ref{L:conrep} along with \eqref{est4-} that
$$
\sup_{\lambda \in (0,1)} \sup_{t \in [0,T_0-h]} \left\| u_\lambda(t+h) - u_\lambda(t) \right\|_H \to 0
$$
as $h \to 0_+$. On the other hand, it follows from \eqref{est1} and (A1) that $(u_\lambda(t))_{\lambda \in (0,1)}$ is precompact in $H$ for each $t \in [0,T_0]$.

Combining all these facts and using Ascoli's compactness lemma (see, e.g.,~\cite[Lemma 1]{Simon}), one can take a sequence $(\lambda_n)$ in $(0,1)$ converging to $0$ such that
\begin{alignat}{4}
u_{\lambda_n} &\to u \quad &&\mbox{ strongly in } C([0,T_0];H), \label{c1}\\
g_{\lambda_n} &\to \xi \quad &&\mbox{ weakly in } L^2(0,T_0;H),\label{c2}\\
\partial \varphi^2_{\lambda_n}(u_{\lambda_n}(\cdot)) &\to \eta \quad &&\mbox{ weakly in } L^2(0,T_0;H),\nonumber\\
(\d/\d t) [k*(u_{\lambda_n} - u_0)] &\to \zeta &&\mbox{ weakly in } L^2(0,T_0;H) \nonumber
\end{alignat}
for some $u \in C([0,T_0];H)$ and $\xi, \eta, \zeta \in L^2(0,T_0;H)$. Moreover, \eqref{c1} along with $u_\lambda(0)=u_0$ also yields $u(0) = u_0$. Recall that the map $\partial \Phi^1: u \mapsto \partial \varphi^1(u(\cdot))$ is maximal monotone in $L^2(0,T_0;H)$ (see~\cite[p.\,25]{HB1}) and every maximal monotone operator is demiclosed (see \S \ref{Ss:subdif}). Since $(u_\lambda,g_\lambda)$ lies on the graph $G(\partial \Phi_1)$ of $\partial \Phi_1$ (equivalently, $(u_\lambda(t),g_\lambda(t)) \in G(\partial \varphi^1)$ for a.e.~$t \in (0,T_0)$), we find from \eqref{c1} and \eqref{c2} that $(u,\xi) \in D(\partial \Phi_1)$, that is,
$$
u(t) \in D(\partial \varphi^1) \quad \mbox{ and } \quad \xi(t) \in \partial \varphi^1(u(t))
$$
for a.e.~$t \in (0,T_0)$. Furthermore, let $J_\lambda : H \to H$ be the resolvent of $\partial \varphi^2$, that is, $J_\lambda = (1 + \lambda \partial \varphi^2)^{-1}$, for $\lambda \in (0,1)$. Note that
$$
\|u_\lambda(t) - J_\lambda u_\lambda(t)\|_H = \lambda \left\| \partial \varphi^2_\lambda(u_\lambda(t)) \right\|_H
$$
for all $t \in [0,T]$. It follows from \eqref{est3} that
$$
u_\lambda - J_\lambda u_\lambda \to 0 \quad \mbox{ strongly in } L^2(0,T_0;H) \ \mbox{ as } \ \lambda \to 0_+,
$$
which along with \eqref{c1} implies
$$
J_{\lambda_n} u_{\lambda_n} \to u \quad \mbox{ strongly in } L^2(0,T_0;H).
$$
Recalling that $(J_\lambda u_\lambda(t), \partial \varphi^2_\lambda(u_\lambda(t)))$ lies on the graph of $\partial \varphi^2$ for a.e.~$t \in (0,T_0)$, we can similarly conclude that
$$
u(t) \in D(\partial \varphi^2) \quad \mbox{ and } \quad \eta(t) \in \partial \varphi^2(u(t))
$$
for a.e.~$t \in (0,T_0)$. Moreover, due to the linear maximal monotonicity (in particular, weak closedness) of $\B$ in $L^2(0,T_0;H)$, we further obtain
$$
u - u_0 \in D(\B) \quad \mbox{ and } \quad \zeta = \B(u-u_0) = \frac{\d}{\d t}[k*(u-u_0)],
$$
the former of which yields the initial condition, that is, $k*(u - u_0) \in W^{1,2}(0,T_0;H)$ and $[k*(u-u_0)](0) = 0$. Finally, from \eqref{est1} along with the lower-semicontinuity of $\varphi^1$ on $H$, one can verify that
$$
\varphi^1(u(\cdot)) \in L^\infty(0,T_0).
$$
This completes the proof.
\end{proof}

\section{Proofs of Theorem \ref{T:SDGE} and Corollary \ref{C:SDGEc}}\label{S:SDGE}

This section is devoted to proving Theorem \ref{T:SDGE} and Corollary \ref{C:SDGEc}. To this end, we first set up the following lemma, which is a variant of Lemma \ref{L:Gron1}.

\begin{lemma}\label{L:Gron2}
Let $S \in (0,\infty)$ and let $\phi \in L^{\infty}(0,S)$ be a nonnegative function. Suppose that there exist constants $\delta > {b} \geq 0$, a nonnegative function $g \in L^1(0,S)$ and a locally bounded Borel measurable function $N : [0,\infty) \to \R$ satisfying
\begin{equation}\label{hyp-M}
N(r) \leq 0 \quad \text{ for } \ r \in [0,\delta]
\end{equation}
such that 
$$
\phi(t) \leq {b} + [g * N(\phi(\cdot))](t) \quad \text{ for a.e.~}\ t \in (0,S).
$$
Then $\|\phi\|_{L^{\infty}(0,S)} \leq {b}$.
\end{lemma}

\begin{proof}
Fix $\epsilon >0$ satisfying $0 < \epsilon < (\delta-{b})/2$ (hence ${b} + 2 \epsilon < \delta$). As in the proof of Lemma \ref{L:Gron1}, we set 
$$
T_{\phi,\epsilon} := \sup\{s \in (0,S) : \|\phi\|_{L^{\infty}(0,s)} \leq {b} + 2\epsilon\} > 0.
$$
From the arbitrariness of $\epsilon>0$, it suffices to show that $T_{\phi,\epsilon} = S$. Suppose to the contrary that $T_{\phi,\epsilon} < S$. Then there exists $t_{\phi,\epsilon} \in (T_{\phi,\epsilon},S)$ such that 
\begin{gather*}
{b}+2 \epsilon < \phi(t_{\phi,\epsilon}) \leq {b} + [g * N(\phi(\cdot))](t_{\phi,\epsilon}),\\
\|N\|_{L^{\infty}(0,R_{\phi})} \int_{0} ^{t_{\phi,\epsilon} - T_{\phi,\epsilon}} g(t) \, \d t < \epsilon, 
\end{gather*}
where $R_{\phi} := \|\phi\|_{L^{\infty}(0,S)}$. As $\phi(t) \leq \delta$ for a.e.~$t \in (0,T_{\phi,\epsilon})$, we find from \eqref{hyp-M} that 
\begin{align*}
{b}+2 \epsilon < \phi(t_{\phi,\epsilon}) 
&\leq {b} + [g * N(\phi(\cdot))](t_{\phi,\epsilon})\\
&= {b} + \int_{0}^{t_{\phi,\epsilon}} g(t_{\phi,\epsilon}-t) N(\phi(t)) \, \d t
\\
&\leq {b} + \int_{T_{\phi,\epsilon}}^{t_{\phi,\epsilon}} g(t_{\phi,\epsilon}-t) N(\phi(t)) \, \d t\\
&\leq {b} + \|N\|_{L^{\infty}(0,R_{\phi})} \int_{0}^{t_{\phi,\epsilon} - T_{\phi,\epsilon}} g(t) \, \d t
< {b} + \epsilon,
\end{align*}
which yields a contradiction. Thus we obtain $S = T_{\phi,\epsilon}$. The desired conclusion follows from the arbitrariness of $\epsilon > 0$. 
\end{proof}

We are in a position to prove Theorem \ref{T:SDGE}.

\begin{proof}[Proof of Theorem {\rm \ref{T:SDGE}}]
Let $T \in (0,\infty)$ be fixed. Set
$$
E_T(u_{0},f) := \varphi^1(u_{0}) + \left\| \ell*\|f(\cdot)\|_{H}^{2} \right\|_{L^\infty(0,T)}
$$ 
and recall the approximate problems \eqref{ee-aprx}, which admit strong solutions $u_\lambda$ on $[0,T]$ such that $u_\lambda \in C([0,T];H)$, $u_\lambda(0) = u_0$ and $\varphi^1(u_\lambda(\cdot)) \in L^\infty(0,T)$. Moreover, let $g_\lambda \in L^2(0,T;H)$ be such that $g_\lambda(t) \in \partial \varphi^1(u_\lambda(t))$ and \eqref{ee-lam} holds for a.e.~$t \in (0,T)$.

From (iv) of Proposition \ref{P:J-Y} and ${\rm (A2)}'$, we have
\begin{align}\label{dphi2-2}
\| \partial \varphi^2_\lambda(u_\lambda(t)) \|_{H}
\leq \|\mathring{ \partial \varphi^2}(u_\lambda(t)) \|_{H}
\leq {\nu_2} \|g_\lambda(t)\|_{H} + M_{2} ( \varphi^1 (u_\lambda(t)) )
\end{align}
for a.e.~$t \in (0,T)$. By the use of \eqref{dphi2-2} instead of \eqref{dphi2}, testing \eqref{ee-lam} by $g_\lambda(t)$ and convolving it with $\ell$, as in the proof of Theorem \ref{T:LE}, we can derive
\begin{align*}
\lefteqn{
\varphi^1(u_\lambda(t)) - \varphi^1(u_0) + \frac{1-{\nu_2}}2 \left( \ell * \|g_\lambda(\cdot)\|_H^2 \right)(t)
}\\
&\leq C_1 \left[ \ell * M_2 \left( \varphi^1(u_\lambda(\cdot)) \right)^2 \right](t) + C_1 \left\| \ell * \| f(\cdot) \|_H^2 \right\|_{L^\infty(0,T)}
\end{align*}
for a.e.~$t \in (0,T)$ and a constant $C_1 = C_1(\nu_2) \geq 0$ (see \eqref{ei1}). Thus one can take a constant $C_3 = C_3(\nu_2) \geq 1$ such that
\begin{align*}
\lefteqn{
\varphi^1(u_\lambda(t)) + \frac{1-{\nu_2}}2 \left( \ell * \|g_\lambda(\cdot)\|_H^2 \right)(t)
}\\
&\leq C_3 E_T(u_0,f) + C_3 \left[ \ell * M_2 \left( \varphi^1(u_\lambda(\cdot)) \right)^2 \right](t)
\end{align*}
for a.e.~$t \in (0,T)$. Here we note from \eqref{A2p-1} that
$$
\|g_\lambda(t)\|_H \geq m_2(\varphi^1(u_\lambda(t)))
$$
for a.e.~$t \in (0,T)$. Hence it follows that
\begin{align}
\lefteqn{
\varphi^1(u_\lambda(t)) + \frac{1-{\nu_2}}4 \left( \ell * \|g_\lambda(\cdot)\|_H^2 \right)(t)
} \nonumber\\
&\leq C_3 E_T(u_0,f) + \left[ \ell * N\left( \varphi^1(u_\lambda(\cdot)) \right) \right](t) \quad \mbox{ for a.e. } \ t \in (0,T), 
\label{sd:ei1}
\end{align}
where $N$ is a function given by
$$
N(r) := C_3 M_2(r)^2 - \frac{1-{\nu_2}}4 m_2(r)^2 \quad \mbox{ for } \ r \geq 0.
$$
Here due to \eqref{M2-growth} one can take a constant $\delta_0 > 0$ small enough depending only on $\nu_2$, $m_2(\cdot)$ and $M_2(\cdot)$ that $N(r) \leq 0$ for any $r \in [0,\delta_0]$. Thanks to Lemma \ref{L:Gron2}, we obtain
\begin{equation}\label{sd:e1}
\sup_{t \in [0,T]} \varphi^1(u_\lambda(t)) \leq C_3 E_T(u_0,f),
\end{equation}
provided that $C_3 E_T(u_0,f) < \delta_0$ (cf.~see~\eqref{est1}). Furthermore, it follows from \eqref{sd:ei1} that
\begin{equation}\label{sd:e:g}
\esssup_{t \in (0,T)} \left( \ell * \|g_\lambda(\cdot)\|_H^2 \right)(t) \leq C,
\end{equation}
which in particular implies
\begin{equation*}
\int^T_0 \|g_\lambda(t)\|_H^2 \, \d t \leq C.
\end{equation*}
Here and henceforth, $C$ denotes a generic constant being independent of $t$ and $\lambda$. Hence combining \eqref{A2p-2} with \eqref{sd:e1} and \eqref{sd:e:g}, we deduce that
\begin{equation*}
\esssup_{t \in (0,T)} \left( \ell * \| \partial \varphi^2_\lambda(u_\lambda(\cdot)) \|_H^2 \right)(t) \leq C,
\end{equation*}
which also yields
\begin{equation*}
\int^T_0 \left\| \partial \varphi^2_\lambda(u_\lambda(t)) \right\|_H^2 \, \d t \leq C,
\end{equation*}
and hence, \eqref{ee-lam} gives
\begin{equation*}
\esssup_{t \in (0,T)} \left( \ell * \left\| \dfrac{\d}{\d t}[ k*(u_\lambda - u_0)](\cdot) \right\|_H^2 \right)(t) \leq C.
\end{equation*}
Due to \eqref{sd:e1} and ${\rm (A1)}'$, one can also deduce that $(u_\lambda(t))_{\lambda \in (0,1)}$ is precompact in $H$ for each $t \in [0,T]$. Hence repeating the same argument as in the proof of Theorem \ref{T:LE}, we can conclude that
$$
u_{\lambda_n} \to u \quad \mbox{ strongly in } C([0,T];H)
$$
for some sequence $\lambda_n \to 0_+$ and $u \in C([0,T];H)$ complying with $u(0) = u_0$, and moreover, the limit $u$ turns out to be a strong solution on $[0,T]$ to \eqref{ee}. Furthermore, \eqref{sd:e1} and the lower-semicontinuity of $\varphi^1$ yield
$$
\varphi^1(u(t)) \leq C_3 E_T(u_0,f)
$$
for all $t \in [0,T]$.

Finally, we show existence of strong solutions on $[0,\infty)$ to the Cauchy problem \eqref{ee} with $T = \infty$ for $u_0 \in D(\varphi^1)$ and $f \in L^2_{\rm loc}([0,\infty);H)$ satisfying $C_3 E_\infty(u_0,f) < \delta_0$. Thanks to Theorem \ref{T:Lip}, for each $\lambda > 0$ the approximate problem \eqref{ee-aprx} with $T = \infty$ admits a unique strong solution $u_\lambda \in L^2_{\rm loc}([0,\infty);H)$ on $[0,\infty)$ satisfying \eqref{regu-Lip} for any $T > 0$. From the arguments so far, for each $N \in \N$, one can inductively take a subsequence $(\lambda^{(N)}_n)_{n \in \N}$ of $(\lambda^{(N-1)}_n)_{n \in \N}$ such that 
$$
u_{\lambda^{(N)}_n} \to u^{(N)} \quad \mbox{ strongly in } C([0,N];H)
$$
for some strong solution $u^{(N)} \in C([0,N];H)$ on $[0,N]$ to \eqref{ee} with $T = N$ and $u^{(N)}(0) = u_0$. Then one observes by definition that the restriction $u^{(N)}|_{[0,N-1]}$ coincides with $u^{(N-1)}$. Due to a diagonal argument, one can take a (not relabeled) subsequence of $(\lambda^{(n)}_n)_{n \in \N}$ such that
$$
u_{\lambda^{(n)}_n} \to u \quad \mbox{ strongly in } C([0,T];H) \ \mbox{ for any } \ T > 0
$$
for some $u \in C([0,\infty);H)$ satisfying $u(0) = u_0$. Moreover, recalling the relation $g_\lambda := f - (\d/\d t)[k*(u_\lambda-u_0)] + \partial \varphi^2_\lambda(u_\lambda(\cdot))$ (see \eqref{ee-lam} with $T = \infty$) along with the uniqueness of $u_\lambda$, one can construct $\xi,\eta \in L^2_{\rm loc}([0,\infty);H)$ satisfying \eqref{EQ} for a.e.~$t \in (0,\infty)$. Indeed, for each $N \in \N$, one can similarly construct $\xi^{(N)}, \eta^{(N)} \in L^2(0,N;H)$ such that \eqref{EQ} holds for a.e.~$t \in (0,N)$, $\xi^{(N)}|_{[0,N-1]} = \xi^{(N-1)}$ and $\eta^{(N)}|_{[0,N-1]} = \eta^{(N-1)}$. Define a strongly measurable function $\xi : (0,\infty) \to H$ as a limit
$$
\xi(t) := \lim_{N \to \infty} \bar{\xi}^{(N)}(t) \quad \mbox{ for a.e.~} \ t \in (0,\infty),
$$
where $\bar{\xi}^{(N)} \in L^2(0,\infty;H)$ is a function given as
$$
\bar{\xi}^{(N)}(t) = \begin{cases}
		\xi^{(N)}(t) &\mbox{ if } \ t \in [0,N],\\
		0 &\mbox{ if } \ t \in (N,\infty)
	       \end{cases}
\quad \mbox{ for } \ t \in [0,\infty).
$$
Then we observe that $\xi|_{[0,N]} = \xi^{(N)}$ for $N \in \N$. Hence $\xi \in L^2_{\rm loc}([0,\infty);H)$. By a similar argument, we can also construct $\eta \in L^2_{\rm loc}([0,\infty);H)$ such that  $\eta|_{[0,N]} = \eta^{(N)}$ for $N \in \N$. Therefore $u$ is a strong solution to \eqref{ee} on $[0,\infty)$ in the sense of Definition \ref{D:sol}, since $u|_{[0,N]} = u^{(N)}$ for each $N$. This completes the proof.
\end{proof}

We now move on to a proof of Corollary \ref{C:SDGEc}.

\begin{proof}[Proof of Corollary {\rm \ref{C:SDGEc}}]
Let $T \in (0,\infty)$ be fixed.  
Due to \eqref{M2-growth-c}, one can take constants $\delta_0, \epsilon > 0$ small enough depending only on $M_3(\cdot)$ that 
$$
0 < M_3(r+\epsilon) \leq \frac 12 \quad \mbox{ for } \ r \in [0,\delta_0].
$$
Testing \eqref{ee-lam} by $g_\lambda(t)$, using Young's inequality, and convolving it with $\ell$, one can deduce that
\begin{align*}
\lefteqn{
\varphi^1(u_\lambda(t)) - \varphi^1(u_0) + \frac 12 \left( \ell * \|g_\lambda(\cdot)\|_H^2 \right)(t)
}\\
&\leq \left( \ell * \|\partial \varphi^2_\lambda(u_\lambda(\cdot))\|_H^2 \right)(t) + \left\| \ell * \| f(\cdot) \|_H^2 \right\|_{L^\infty(0,T)}
\end{align*}
for a.e.~$t \in (0,T)$. From (iv) of Proposition \ref{P:J-Y} and \eqref{A2p-2c} of ${\rm (A2)}'_c$, we have
$$
\| \partial \varphi^2_\lambda(u_\lambda(t)) \|_{H}
\leq \|\mathring{ \partial \varphi^2}(u_\lambda(t)) \|_{H}
\leq \|g_\lambda(t)\|_{H} M_3 \big( \varphi^1 (u_\lambda(t)) \big),
$$
which yields
\begin{align}\label{dphi2-2c}
\| \partial \varphi^2_\lambda(u_\lambda(t)) \|_{H} M_3 \big( \varphi^1 (u_\lambda(t)) + \epsilon \big)^{-1}
\leq \|g_\lambda(t)\|_{H}
\end{align}
for a.e.~$t \in (0,T)$. Using \eqref{dphi2-2c}, we obtain
\begin{align*}
\lefteqn{
\varphi^1(u_\lambda(t)) + \frac 14 \left( \ell * \|g_\lambda(\cdot)\|_H^2 \right)(t)
}\\
& \leq E_T(u_0,f) + \left[ \ell * N\left( \varphi^1(u_\lambda(\cdot)) \right) \right](t)
\end{align*}
for a.e.~$t \in (0,T)$. Here $N$ is a function given by
$$
N(r) := \begin{cases}
	 0 &\mbox{ if } \ 1 - (1/4) M_3(r+\epsilon)^{-2} \leq 0,\\
	 \|\partial \varphi^2_\lambda(u_\lambda(\cdot))\|_{L^\infty(0,T;H)}^2  &\mbox{ if } \ 1 - (1/4) M_3(r+\epsilon)^{-2} > 0
	\end{cases}
$$
for $r \geq 0$. Then we note that $N(r) \leq 0$ for any $r \in [0,\delta_0]$. Thanks to Lemma \ref{L:Gron2}, we obtain
\begin{equation*}
\esssup_{t \in (0,T)} \varphi^1(u_\lambda(t)) \leq E_T(u_0,f),
\end{equation*}
provided that $E_T(u_0,f) < \delta_0$. The rest of the proof runs as before.
\end{proof}

\section{Proof of Theorem \ref{T:GE}}\label{S:GE}

In this section, we prove Theorem \ref{T:GE}. 

\begin{proof}[Proof of Theorem {\rm \ref{T:GE}}] 
Let $T \in (0,\infty)$ be fixed. In this proof, we consider the following (slightly different) approximate problems:
\begin{equation}
\left.
\begin{aligned}
\lambda \frac{\d u_\lambda}{\d t}(t) + \frac{\d}{\d t} \left[ k * (u_\lambda - u_0) \right] (t) + \partial \varphi^1(u_\lambda(t)) - \partial \varphi^2_\lambda(u_\lambda(t)) &\ni f(t),\\
0 < t &< T,\\
u_\lambda(0) &= u_0,
\end{aligned}
\right\} \label{ee-aprx2}
\end{equation}
where $\varphi^2_\lambda$ stands for the Moreau--Yosida regularization of $\varphi^2$, for $\lambda \in (0,1)$. For $u_0 \in D(\varphi^1)$, $f \in W^{1,2}(0,T;H)$ and $\lambda \in (0,1)$, one can prove existence of a unique strong solution $u_\lambda \in W^{1,2}(0,T;H)$ on $[0,T]$ to the Cauchy problem \eqref{ee-aprx} as in~\cite{A19} (see \S \ref{Ss:aprx2} in Appendix). Testing \eqref{ee-aprx2} by $(\d u_\lambda/\d t)(t)$ and employing Proposition \ref{P:std_chain}, we have
\begin{align*}
\lefteqn{
\lambda \left\| \dfrac{\d u_\lambda}{\d t}(t) \right\|_H^2 + \left( \dfrac{\d}{\d t} \left[ k* (u_\lambda - u_0) \right](t), \dfrac{\d u_\lambda}{\d t}(t) \right)_H 
}\\
&\quad + \dfrac{\d}{\d t} \left[ \varphi^1(u_\lambda(t)) - \varphi^2_\lambda(u_\lambda(t)) \right]\\
&= \left(f(t), \dfrac{\d u_\lambda}{\d t}(t) \right)_H \\
&= \dfrac{\d}{\d t} (f(t), u_\lambda(t))_H - \left( \dfrac{\d f}{\d t}(t), u_\lambda(t) \right)_H
\end{align*}
for a.e.~$t \in (0,T)$. Note that
\begin{align*}
\lefteqn{
\int^t_0 \left( \dfrac{\d}{\d t} \left[ k* (u_\lambda - u_0) \right](\tau), \dfrac{\d u_\lambda}{\d t}(\tau) \right)_H \, \d \tau}\\
&\geq \dfrac 12 \left( \ell * \left\| \dfrac{\d}{\d t} \left[ k* (u_\lambda - u_0) \right] (\cdot) \right\|_H^2 \right)(t)
\end{align*}
for all $t \in [0,T]$ (see Proposition \ref{P:AB}). Integrating both sides over $(0,t)$, we find that
\begin{align*}
\lefteqn{
\lambda \int^t_0 \left\| \dfrac{\d u_\lambda}{\d t}(\tau) \right\|_H^2 \, \d \tau + \dfrac 12 \left( \ell * \left\| \dfrac{\d}{\d t} \left[ k* (u_\lambda - u_0) \right](\cdot) \right\|_H^2 \right)(t)
}\\
&\quad + \varphi^1(u_\lambda(t)) - \varphi^2_\lambda(u_\lambda(t))\\
&\leq \varphi^1(u_0) - \varphi^2_\lambda(u_0) + (f(t), u_\lambda(t))_H - (f(0), u_0)_H \\
&\quad - \int^t_0 \left( \dfrac{\d f}{\d t}(\tau), u_\lambda(\tau) \right)_H \, \d \tau
\end{align*}
for $t \in [0,T]$. Using (A3), we observe that
\begin{align*}
\varphi^1(u_\lambda(t)) - \varphi^2_\lambda(u_\lambda(t))
&\geq \varphi^1(u_\lambda(t)) - \varphi^2(u_\lambda(t))\\
&\geq (1-{\nu_3}) \varphi^1(u_\lambda(t)) - c_1(\|u_\lambda(t)\|_H^r + 1)
\end{align*}
for $t \in [0,T]$. Moreover, exploiting Jensen's inequality, we note that
\begin{align*}
\lefteqn{
\left( \ell * \left\| \dfrac{\d}{\d t} \left[ k* (u_\lambda - u_0) \right](\cdot) \right\|_H^2 \right)(t)
}\\
&= \int^t_0 \ell(t-s) \left\| \dfrac{\d}{\d t} \left[ k* (u_\lambda - u_0) \right](s) \right\|_H^2 \, \d s\\
&\geq \|\ell\|_{L^1(0,T)}^{-1} \left\| \left( \ell * \dfrac{\d}{\d t} \left[ k* (u_\lambda - u_0) \right] \right)(t) \right\|_H^2\\
&\hspace{-2.2mm}\stackrel{\eqref{l*d-alpha}}= \|\ell\|_{L^1(0,T)}^{-1} \|u_\lambda(t)-u_0\|_H^2
\end{align*}
for $t \in [0,T]$. Thus we obtain
\begin{align*}
\lefteqn{
\lambda \int^t_0 \left\| \dfrac{\d u_\lambda}{\d t}(\tau) \right\|_H^2 \, \d \tau + \dfrac 14 \left( \ell * \left\| \dfrac{\d}{\d t} \left[ k* (u_\lambda - u_0) \right](\cdot) \right\|_H^2 \right)(t)
}\\
&\quad + \frac 14 \|\ell\|_{L^1(0,T)}^{-1} \|u_\lambda(t)-u_0\|_H^2 + (1-{\nu_3})\varphi^1(u_\lambda(t))\\
&\leq \varphi^1(u_0) - \varphi^2_\lambda(u_0) + c_1(\|u_\lambda(t)\|_H^r + 1) + (f(t), u_\lambda(t))_H - (f(0), u_0)_H \\
&\quad - \int^t_0 \left( \dfrac{\d f}{\d t}(\tau), u_\lambda(\tau) \right)_H \, \d \tau
\end{align*}
for $t \in [0,T]$. Therefore exploiting Young's inequality, one finds that
\begin{align*}
\lefteqn{
\lambda \int^t_0 \left\| \dfrac{\d u_\lambda}{\d t}(\tau) \right\|_H^2 \, \d \tau + \dfrac 14 \left( \ell * \left\| \dfrac{\d}{\d t} \left[ k* (u_\lambda - u_0) \right](\cdot) \right\|_H^2 \right)(t)
}\\
&\quad + \frac 18 \|\ell\|_{L^1(0,T)}^{-1} \|u_\lambda(t)-u_0\|_H^2 + (1-{\nu_3})\varphi^1(u_\lambda(t))\\
&\leq \varphi^1(u_0) - \varphi^2_\lambda(u_0) + C\left( \|f(t)\|_H^2 + \|f(0)\|_H^2 + \|u_0\|_H^2 + 1 \right)\\
&\quad + \int^t_0 \left\|\dfrac{\d f}{\d t}(\tau)\right\|_H \|u_\lambda(\tau)\|_H \, \d \tau
\end{align*}
for $t \in [0,T]$. Here and henceforth, $C$ denotes a generic constant being independent of $t$ and $\lambda$. Since $\varphi^2_\lambda(u_0)$ is bounded as $\lambda \to 0_+$, using Gronwall's inequality, we conclude that
\begin{align}
\sup_{t \in [0,T]} \left( \varphi^1(u_\lambda(t)) + \|u_\lambda(t)\|_H^2 \right) &\leq C,\label{ge:e1}\\
\lambda \int^T_0 \left\| \dfrac{\d u_\lambda}{\d t}(\tau) \right\|_H^2 \, \d \tau &\leq C,\nonumber\\
\sup_{t \in [0,T]} \left( \ell * \left\| \dfrac{\d}{\d t} \left[ k* (u_\lambda - u_0) \right](\cdot) \right\|_H^2 \right)(t) &\leq C \label{ge:e1.5}
\end{align}
for $\lambda \in (0,1)$. Furthermore, since $\varphi^1$ is affinely bounded from below, we infer that
$$
\sup_{t \in [0,T]} |\varphi^1(u_\lambda(t))| \leq C.
$$
Hence as in the proof of Theorem \ref{T:LE}, we can deduce from \eqref{ge:e1.5} along with Lemma \ref{L:conrep}, (A1) and Ascoli's theorem that
$$
u_{\lambda_n} \to u \quad \mbox{ strongly in } C([0,T];H)
$$
for some sequence $\lambda_n \to 0_+$, and therefore, we have $u(0) = u_0$. Moreover, the convolution of \eqref{ge:e1.5} with $k$ also yields
\begin{equation*}
 \int^T_0 \left\| \dfrac{\d}{\d t} \left[ k*(u_\lambda-u_0)\right](t)\right\|_H^2 \, \d t \leq C.
\end{equation*}
Furthermore, it follows from (A2) and \eqref{ge:e1} that
\begin{equation}\label{dphi2-ge}
\left\|\partial \varphi^2_\lambda (u_\lambda(t))\right\|_H
 \leq \nu_1 \|g_\lambda(t)\|_H + C
\end{equation}
for a.e.~$t \in (0,T)$. Testing \eqref{ee-aprx2} by $g_\lambda(t)$, one finds that
\begin{align*}
\lefteqn{
\lambda \dfrac{\d}{\d t} \varphi^1(u_\lambda(t)) + \left( \dfrac{\d}{\d t} \left[ k* (u_\lambda - u_0) \right](t), g_\lambda(t) \right)_H + \|g_\lambda(t)\|_H^2
}\\
&\leq \left\|\partial \varphi^2_\lambda(u_\lambda(t)) \right\|_H \|g_\lambda(t)\|_H + \|f(t)\|_H \|g_\lambda(t)\|_H\\
&\leq \nu_1 \|g_\lambda(t)\|_H^2 + C \|g_\lambda(t)\|_H + \|f(t)\|_H \|g_\lambda(t)\|_H
\end{align*}
for a.e.~$t \in (0,T)$. With the aid of Young's inequality, we have
\begin{align*}
\lefteqn{
\lambda \dfrac{\d}{\d t} \varphi^1(u_\lambda(t)) + \left( \dfrac{\d}{\d t} \left[ k* (u_\lambda - u_0) \right](t), g_\lambda(t) \right)_H + \frac{1-\nu_1}2 \|g_\lambda(t)\|_H^2
}\\
&\leq C \left( \|f(t)\|_H^2 + 1 \right)
\end{align*}
for a.e.~$t \in (0,T)$. Integrating both sides over $(0,T)$ and using (i) of Proposition \ref{P:frac_chain}, we conclude that
\begin{align*}
\lefteqn{
\lambda \varphi^1(u_\lambda(t)) + \left[ k * \left( \varphi^1(u_\lambda(\cdot)) - \varphi^1(u_0)\right)\right](t) + \frac{1-\nu_1}2 \int^t_0 \|g_\lambda(t)\|_H^2 \, \d t
}\\
&\leq C \left( \int^T_0 \|f(t)\|_H^2 \, \d t + 1 \right) \quad \mbox{ for } \ t \in [0,T],
\end{align*}
which along with \eqref{dphi2-ge} implies
\begin{equation*}
 \int^T_0 \|g_\lambda(t)\|_H^2 \, \d t + \int^T_0 \left\|\partial \varphi^2_\lambda(u_\lambda(t))\right\|_H^2 \, \d t \leq C.
\end{equation*}
The rest of the proof runs as in \S \ref{S:SDGE}.
\end{proof}

\section{Applications}\label{S:app}

In this section, we apply the abstract theory developed so far in the present paper to the Cauchy--Dirichlet problem for $p$-Laplace subdiffusion equations with blow-up terms of the form,
\begin{alignat}{4}
\partial_t^\alpha (u-u_0) - \Delta_p u - |u|^{q-2}u &= f \ &&\mbox{ in } \ \Omega \times (0,\infty),\label{pde2}\\
u &= 0 &&\mbox{ on } \partial \Omega \times (0,\infty),\label{bc2}
\end{alignat}
where $0 < \alpha < 1$, $1 < p,q < \infty$, $\Omega$ is a bounded domain of $\R^d$ with smooth boundary $\partial \Omega$ and $u_0 = u_0(x)$ and $f = f(x,t)$ are prescribed. Moreover, $\partial_t^\alpha (u-u_0) := \partial_t [k_\alpha * (u-u_0)]$ denotes the Riemann--Liouville derivative of $u-u_0$ (i.e., Caputo derivative of $u$) of order $0 < \alpha < 1$ and $\Delta_p$ is the so-called \emph{$p$-Laplacian} given as
$$
\Delta_p u = \nabla \cdot \left(|\nabla u|^{p-2}\nabla u\right).
$$
Here and henceforth, we are concerned with $L^2$-solutions to \eqref{pde2}, \eqref{bc2} defined as follows:
\begin{definition}[$L^2$-solutions to \eqref{pde2}, \eqref{bc2}]\label{D:pde-sol}
For each $S \in (0,\infty)$, a function $u \in L^2(0,S;L^2(\Omega))$ is called an \emph{$L^2$-solution} on $[0,S]$ to the Cauchy--Dirichlet problem \eqref{pde2}, \eqref{bc2}, if the following conditions are all satisfied\/{\rm :}
\begin{enumerate}
 \item[\rm (i)] It holds that $k_{\alpha}*(u-u_0) \in W^{1,2}(0,S;L^2(\Omega))$, $[k_{\alpha}*(u-u_0)](0) = 0$ and $u(t) \in W^{1,p}_0(\Omega) \cap L^q(\Omega)$ for a.e.~$t \in (0,S)$.
 \item[\rm (ii)] It holds that $\Delta_p u, |u|^{q-2}u \in L^2(0,S;L^2(\Omega))$ and
$$
\dfrac{\d}{\d t} \left[ k_\alpha * (u - u_0) \right](\cdot,t) - \Delta_p u(\cdot,t) - |u|^{q-2}u(\cdot,t) = f(\cdot,t) \ \mbox{ in } L^2(\Omega)
$$
for a.e.~$t \in (0,S)$.
\end{enumerate}
Moreover, a function $u \in L^2_{\rm loc}([0,\infty);L^2(\Omega))$ is called an \emph{$L^2$-solution} on $[0,\infty)$ to the Cauchy--Dirichlet problem \eqref{pde2}, \eqref{bc2}, if the restriction $u|_{[0,S]}$ of $u$ is an $L^2$-solution on $[0,S]$ to \eqref{pde2}, \eqref{bc2} for every $S > 0$.
\end{definition}

We set $H = L^2(\Omega)$ and define $\varphi^1, \varphi^2 :H \to [0,\infty]$ by
\begin{align*}
 \varphi^1(w) &= \begin{cases}
		  \frac 1p \int_\Omega |\nabla w|^p \, \d x &\mbox{ if } \ w \in W^{1,p}_0(\Omega),\\
		  \infty &\mbox{ otherwise,}
		 \end{cases}\\
 \varphi^2(w) &= \begin{cases}
		  \frac 1q \int_\Omega |w|^q \, \d x &\mbox{ if } \ w \in L^q(\Omega),\\
		  \infty &\mbox{ otherwise}
		 \end{cases}
\end{align*}
for $w \in H = L^2(\Omega)$. Then $\varphi^1$ and $\varphi^2$ turn out to be proper, lower-semicontinuous and convex in $H$. Moreover, $\partial \varphi^1(w)$ and $\partial \varphi^2(w)$ coincide with $-\Delta_p w$ equipped with the homogeneous Dirichlet boundary condition in the distributional sense for $w \in D(\partial \varphi^1) = \{w \in W^{1,p}_0(\Omega) \cap L^2(\Omega) \colon \Delta_p w \in L^2(\Omega)\}$ and $|w|^{q-2}w$ for $w \in D(\partial \varphi^2) = \{w \in L^2(\Omega) \cap L^q(\Omega) \colon |w|^{q-2}w \in L^2(\Omega)\}$, respectively. Hence the Cauchy--Dirichlet problem \eqref{pde2}, \eqref{bc2} is reduced into the abstract Cauchy problem,
\begin{equation}\label{ee-pde}
\dfrac{\d}{\d t} \left[ k_{\alpha} * (u-u_0) \right](t) + \partial \varphi^1(u(t)) - \partial \varphi^2(u(t)) = f(t) \ \mbox{ in } H
\end{equation}
for $t \in (0,\infty)$. 

We now set up the following:
\begin{lemma}[Calder\'on--Zygmund estimates for $p$-Laplacian]\label{L:pCZ}
Let $\Omega$ be a bounded $C^{1,1}$ domain of $\R^d$ with boundary $\partial \Omega$. Let $p \in (1,\infty)$ be such that $p \geq 2d/(d+2)$. Let $r = 2^*(p-1)$ for $d \geq 3$ and $r \in (p,\infty)$ arbitrarily for $d \leq 2$. Then for each $f \in L^2(\Omega)$, the Dirichlet problem
\begin{equation}\label{p-ell-eq}
-\Delta_p u = f \ \mbox{ in } \Omega, \quad u = 0 \ \mbox{ on } \partial \Omega
\end{equation}
admits a unique weak solution $u \in W^{1,p}_0(\Omega)$ such that $\nabla u \in L^{r}(\Omega;\R^d)$. Moreover, there exists a constant $C \geq 0$ depending only on $d$, $p$ and $\partial \Omega$ such that 
$$
\|\nabla u\|_{L^{r}(\Omega;\R^d)}^{p-1} \leq C \|f\|_{L^2(\Omega)}
$$
for any $f \in L^2(\Omega)$. Here $u \in W^{1,p}_0(\Omega)$ denotes the unique weak solution to \eqref{p-ell-eq}.
\end{lemma}

\begin{proof}
We only treat the case where $d \geq 3$, but the case where $d \leq 2$ can also be treated analogously with the continuous embedding $H^1(\Omega) \hookrightarrow L^\rho(\Omega)$ for any $1 \leq \rho < \infty$. Let $\phi \in H^2(\Omega) \cap H^1_0(\Omega)$ be the unique strong solution to the Dirichlet problem,
$$
\Delta \phi = f \ \mbox{ in } \Omega, \quad \phi = 0 \ \mbox{ on } \partial \Omega.
$$
Then it follows from \eqref{p-ell-eq} that
\begin{align*}
 - \Delta_p u &= f = \mathrm{div} \left(\nabla \phi \right) = \mathrm{div} \left( |F|^{p-2}F \right), \quad u = 0 \ \mbox{ on } \partial \Omega,
\end{align*}
where $F \in L^{2^*(p-1)}(\Omega;\R^d)$ is given by
$$
F := |\nabla \phi|^{p'-2}\nabla \phi.
$$
Here we note that $p \leq 2^*(p-1)$ is equivalent to $p \geq 2d/(d+2)$. Therefore by virtue of nonlinear Calder\'on--Zygmund estimates established in~\cite[Theorem 1.6]{KiZh01}, we obtain $\nabla u \in L^{2^*(p-1)}(\Omega;\R^d)$ and
$$
\|\nabla u\|_{L^{2^*(p-1)}(\Omega;\R^d)} \leq C \|F\|_{L^{2^*(p-1)}(\Omega;\R^d)} = C \|\nabla \phi\|_{L^{2^*}(\Omega;\R^d)}^{p'-1}
$$
for some constant $C \geq 0$ independent of $f$. Furthermore, we observe that
\begin{align*}
 \|\nabla \phi\|_{L^{2^*}(\Omega;\R^d)}
 \lesssim \|\nabla \phi\|_{H^1(\Omega;\R^d)}
 \lesssim \|\phi\|_{H^2(\Omega)} \lesssim \|\Delta \phi\|_{L^2(\Omega)} = \|f\|_{L^2(\Omega)}
\end{align*}
due to (linear) Calder\'on--Zygmund estimates. Here $L \lesssim R$ means $L \leq C R$ for some constant $C \geq 0$ independent of $f$. Thus the desired conclusion follows.
\end{proof}

We first apply Theorem \ref{T:LE} to obtain local (in time) existence of $L^2$-solutions to the Cauchy--Dirichlet problem \eqref{pde2}, \eqref{bc2}.

\begin{theorem}[Local existence]\label{T:pLE}
Suppose that
\begin{equation}\label{pq-hyp}
p > \dfrac{2d}{d+2} \quad \mbox{ and } \quad q < p^* := \frac{dp}{(d-p)_+}
\end{equation}
and let $T \in (0,\infty)$. For every $u_0 \in W^{1,p}_0(\Omega)$ and $f \in L^{2}(0,T;L^2(\Omega))$ satisfying 
$$
\esssup_{t \in (0,T)} \left( \int^t_0 k_{1-\alpha}(t-s) \|f(\cdot,s)\|_{L^2(\Omega)}^2 \, \d s \right) < \infty,
$$
there exists $T_0 \in (0,\infty)$ such that the Cauchy--Dirichlet problem \eqref{pde2}, \eqref{bc2} admits an $L^2$-solution $u \in C([0,T_0];L^2(\Omega)) \cap L^\infty(0,T_0;W^{1,p}_0(\Omega))$ on $[0,T_0]$ such that $u(\cdot,0) = u_0$.
\end{theorem}

\begin{proof}
We note that $k = k_\alpha$ satisfies (K) with $\ell = k_{1-\alpha}$ (see \S \ref{S:Intro}). We next check (A1) and (A2) for $\varphi^1$ and $\varphi^2$ defined above. First of all, (A1) follows immediately from Rellich--Kondrachov's theorem (i.e., $W^{1,p}_0(\Omega)$ is compactly embedded in $L^2(\Omega)$, provided that $p > 2d/(d+2)$). We next check (A2). In case $2(q-1) \leq p^*$, one finds that
$$
\left\|\partial \varphi^2(w) \right\|_H
= \left\| |w|^{q-2}w \right\|_{L^2(\Omega)} \leq C \| \nabla w \|_{L^p(\Omega;\R^d)}^{q-1} = C \{ \varphi^1(w) \}^{(q-1)/p}
$$
for $w \in W^{1,p}_0(\Omega)$. In case $2(q-1) > p^*$, let us consider $d \geq 3$. As $q \leq p^*$ and $2d/(d+2) < p$, we see that $2(q-1) \leq 2(p^* - 1) < [2^*(p-1)]^* = 2d(p-1)/(d-2p)_+$. Hence by virtue of Gagliardo--Nirenberg's inequalities along with Lemma \ref{L:pCZ}, one can take a constant $\theta \in (0,1)$ such that
\begin{align}
\left\|\partial \varphi^2(w) \right\|_H
&= \left\| w \right\|_{L^{2(q-1)}(\Omega)}^{q-1} \nonumber \\
&\leq C \| \nabla w \|_{L^{2^*(p-1)}(\Omega;\R^d)}^{\theta(q-1)} \| w \|_{L^{p^*}(\Omega)}^{(1-\theta)(q-1)} \nonumber\\
&\leq C \|\Delta_p w\|_{L^2(\Omega)}^{\theta(q-1)/(p-1)}\{ \varphi^1(w) \}^{(1-\theta)(q-1)/p} \label{a2c*}
\end{align}
for $w \in D(\partial \varphi^1)$. Here one can observe that
$$
\theta (q-1)/(p-1) < 1 \quad \mbox{ if and only if } \quad q < p^*.
$$ 
Hence (A2) follows from \eqref{pq-hyp}, and analogously for $d \leq 2$. Thus applying Theorem \ref{T:LE} to \eqref{ee-pde}, we obtain the desired conclusion.
\end{proof}

In case $p < q$, the following theorem assures global (in time) existence of $L^2$-solutions to \eqref{pde2}, \eqref{bc2} for small data for $q \in (p,p^*]$ even including the Sobolev critical exponent $p^*$.

\begin{theorem}[Global existence for small data and $p < q$]\label{T:p<q:SDGE}
In addition to $p < q$, suppose that
\begin{equation*}
p > \dfrac{2d}{d+2} \quad \mbox{ and } \quad q \leq p^* = \frac{dp}{(d-p)_+}.  
\end{equation*}
Then there exists $\delta_0 > 0$ such that, for every $u_0 \in W^{1,p}_0(\Omega)$ and $f \in L^{2}_{\rm loc}([0,\infty);L^2(\Omega))$, the Cauchy--Dirichlet problem \eqref{pde2}, \eqref{bc2} possesses an $L^2$-solution $u$ on $[0,\infty)$ satisfying
\begin{equation*}
u \in C([0,\infty);L^2(\Omega)) \cap L^\infty(0,\infty;W^{1,p}_0(\Omega)), \quad u(\cdot,0) = u_0,
\end{equation*}
provided that 
$$
\|u_0\|_{W^{1,p}_0(\Omega)}^p + \esssup_{t > 0} \left( \int^t_0 k_{1-\alpha}(t-s) \|f(\cdot,s)\|_{L^2(\Omega)}^2 \, \d s \right) < \delta_0.
$$
\end{theorem}

\begin{proof}
We start with the case where $q < p^*$. We observe that
\begin{align*}
\left\|\partial \varphi^1(w)\right\|_H 
&= \|\Delta_p w\|_H\\
&\geq C\|\Delta_p w\|_{W^{-1,p'}(\Omega)}\\
&= C \|\nabla w\|_{L^p(\Omega;\R^d)}^{p-1} = C \{ \varphi^1(w) \}^{(p-1)/p}
\end{align*}
for $w \in D(\partial \varphi^1)$. Hence \eqref{A2p-1} of ${\rm (A2)}'$ holds with $m_2(r) := C r^{(p-1)/p} > 0$ for $r > 0$. In case $2(q-1) \leq p^*$,  ${\rm (A2)}'$ holds with $M_2(r) := C r^{(q-1)/p}$ for $r \geq 0$. Hence $M_2(r)/m_2(r) = C r^{(q-p)/p} \to 0$ as $r \to 0_+$ because of $q > p$. In case $2(q-1) > p^*$, let us consider $d \geq 3$. Then one can check \eqref{A2p-2} of ${\rm (A2)}'$ with
$$
M_2(r) = C r^{\frac 1p \frac{(1-\theta)(q-1)}{1-\theta(q-1)/(p-1)}} \quad \mbox{ for } \ r \geq 0,
$$ 
where $\theta \in (0,1)$ is same as in the proof of Theorem \ref{T:pLE}. Then we find from $p < q < p^*$ that
$$
\lim_{r \to 0_+} \frac{M_2(r)}{m_2(r)} = C \lim_{r \to 0_+} r^{\frac 1p \frac{q-p}{1-\theta(q-1)/(p-1)}} = 0,
$$
which yields ${\rm (A2)}'$. Thus the desired conclusion follows from Theorem \ref{T:SDGE}. As for the Sobolev critical case where $q = p^*$, we can still use \eqref{a2c*}, and moreover, we observe that
$$
\theta (q-1)/(p-1) = 1.
$$
Hence ${\rm (A2)}'_c$ follows from \eqref{a2c*}. Furthermore, the case $d \leq 2$ can also be treated analogously. Therefore the same conclusion follows from Corollary \ref{C:SDGEc}.
\end{proof}

In case $p > q$, one can apply Theorem \ref{T:GE} to prove existence of $L^2$-solutions on $[0,\infty)$ to \eqref{pde2}, \eqref{bc2}. 
\begin{theorem}[Global existence for $p > q$]\label{T:p>q}
In addition to $p > q$, suppose that \eqref{pq-hyp} holds true. For every $u_0 \in W^{1,p}_0(\Omega)$ and $f \in W^{1,2}_{\rm loc}([0,\infty);L^2(\Omega))$, the Cauchy--Dirichlet problem \eqref{pde2}, \eqref{bc2} admits an $L^2$-solution $u$ on $[0,\infty)$ satisfying
\begin{equation*}
u \in C([0,\infty);L^2(\Omega)) \cap L^\infty_{\rm loc}([0,\infty);W^{1,p}_0(\Omega)), \quad u(\cdot,0) = u_0.
\end{equation*}
\end{theorem}

\begin{proof}
We have already checked (A1) and (A2) in the proof of Theorem \ref{T:pLE}. Moreover, we find that
$$
\varphi^2(w) = \dfrac 1q \|w\|_{L^q(\Omega)}^q \leq C \|\nabla w\|_{L^p(\Omega;\R^d)}^{q} = C \{ \varphi^1(w) \}^{q/p}
$$
for $w \in W^{1,p}_0(\Omega)$, which along with the assumption $p > q$ yields (A3). Thus applying Theorem \ref{T:GE} to \eqref{ee-pde}, we obtain the desired conclusion.
\end{proof}

We close this section with the following

\begin{remark}[On classical $p$-Laplace Fujita equation]\label{R:pCZ2}
The nonlinear Calder\'on--Zygmund estimates (see Lemma \ref{L:pCZ}) also enable us to upgrade all the existence results obtained in~\cite{Ishii77,Otani79,Otani82,Otani84} concerning \eqref{pde-c} to the Sobolev subcritical range $1 < q < p^*$ even for $p \neq 2$. Furthermore, following the same argument as in the proof of Corollary \ref{C:SDGEc}, one may also prove global (in time) existence of $L^2$-solutions for small data even for the Sobolev critical exponent $q = p^*$.
\end{remark}

\appendix

\section{Existence of strong solutions to \eqref{ee-aprx2}}\label{Ss:aprx2}

This appendix is devoted to exhibiting an outline of a proof for the existence of strong solutions to the Cauchy problem \eqref{ee-aprx2}. We also refer the reader to~\cite{A19}, where the case $\lambda = 0$ is treated and which may also be useful to handle the present case $\lambda > 0$.\footnote{The present case $\lambda > 0$ is much simpler to the case $\lambda = 0$. Actually, due to the presence of the first-order derivative, one can get better a priori estimates for (approximate) solutions in a standard manner.} Let $T \in (0,\infty)$ be fixed. Then the Cauchy problem \eqref{ee-aprx2} is rewritten as
\begin{equation}\label{aux}
\lambda \A(u - u_0) + \B(u - u_0) + \partial \Phi(u) + B(u(\cdot)) \ni f
\end{equation}
for \emph{positive} $\lambda \in (0,1)$, $u_0 \in D(\varphi)$, $f \in L^2(0,T;H)$ and $B : H \to H$ satisfying \eqref{Lip} (see Subsections \ref{Ss:subdif} and \ref{Ss:fracderi} for the definition of $\A$, $\B$ and $\Phi$). To prove the existence of solutions to \eqref{aux}, we set $\mathcal{X} = C([0,T];H)$ equipped with the norm $\|u\|_\mathcal{X} := \sup_{t \in [0,T]} \e^{-\omega t} \|u(t)\|_H$ for some $\omega > 0$ which will be determined later and define a mapping $S : \mathcal{X} \to \mathcal{X}$ given by
$$
S(v) = u \quad \mbox{ for } \ v \in \mathcal{X},
$$
where $u \in D(\A) + u_0$ is the unique solution of
\begin{equation}\label{aux-p}
\lambda \A(u - u_0) + \B(u - u_0) + \partial \Phi(u) \ni f - B(v(\cdot)).
\end{equation}
Hence it suffices to find a fixed point of $S$. We first show that, for each $v \in L^2(0,T;H)$, \eqref{aux-p} admits a unique solution $u \in D(\A) + u_0$. To this end, let us introduce
\begin{equation}\label{aux-p-aprx}
\lambda \A(u_\mu - u_0) + \B(u_\mu - u_0) + \partial \Phi_\mu(u) \ni f - B(v(\cdot)).
\end{equation}
Then for each $\mu \in (0,1)$ and $v \in L^2(0,T;H)$, the Cauchy problem \eqref{aux-p-aprx} always possesses a unique solution $u_\mu \in D(\A) + u_0$. Indeed, since $\lambda \A + \B$ is maximal monotone in $L^2(0,T;H)$ (see Proposition \ref{P:AB}) and $D(\partial \Phi_\mu) = L^2(0,T;H)$, one finds that $\lambda \A + \B + \partial \Phi_\mu(\,\cdot\, + u_0)$ is surjective on $L^2(0,T;H)$ (see, e.g.,~\cite{HB1},~\cite{B}). Hence as $f - B(v(\cdot)) \in L^2(0,T;H)$, there exists $u_\mu \in D(\A) + u_0$ such that \eqref{aux-p-aprx} holds. As in the proof of Theorem 2.3 of~\cite{A19}, testing \eqref{aux-p-aprx} by $\d u/\d t$ and noting that
\begin{align}
\frac 12 \frac{\d}{\d t} \|w(t)\|_H^2 
&= \left( \dfrac{\d w}{\d t}(t), w(t) \right)_H \nonumber\\
&\leq \left\| \dfrac{\d w}{\d t}(t) \right\|_H \|w(t)\|_H \quad \mbox{ for a.e. } \ t \in (0,T)\label{Lei}
\end{align}
for $w \in W^{1,2}(0,T;H)$, we can derive
$$
\sup_{\mu \in (0,1)} \left( \lambda \int^t_0 \left\| \dfrac{\d u_\mu}{\d t}(\tau) \right\|_H^2 \, \d \tau + \varphi(u_\mu(t)) \right) \leq C
$$
for some $C \geq 0$, which may depend on $\lambda > 0$ but is independent of $\mu \in (0,1)$. Furthermore, as $\lambda$ is positive, one can prove that $(u_\mu)$ forms a Cauchy sequence in $C([0,T];H)$ by using \eqref{Lei} again. Thus we can verify that
\begin{alignat*}{4}
 u_{\mu_n} &\to u \quad &&\mbox{ strongly in } C([0,T];H),\\
 u_{\mu_n} &\to u \quad &&\mbox{ weakly in } W^{1,2}(0,T;H)
\end{alignat*}
for some $\mu_n \to 0_+$ and $u \in W^{1,2}(0,T;H)$ satisfying $u(0) = u_0$, and then, as in the proof of Theorem 2.3 of~\cite{A19}, the limit $u$ turns out to be a solution to \eqref{aux-p}.
 
We next claim that choosing $\omega > 0$ large enough, one can take a constant $\kappa \in (0,1)$ such that
\begin{equation}\label{contraction}
 \|S(v_1)-S(v_2)\|_\mathcal{X} \leq \kappa \|v_1-v_2\|_\mathcal{X} \quad \mbox{ for } \ v_1,v_2 \in \mathcal{X}.
\end{equation}
Then by virtue of Banach's contraction mapping principle, we can deduce that $S$ admits a unique fixed point $u \in \mathcal{X}$, i.e., $S(u) = u$, which implies that $u$ is a solution to \eqref{aux}. Moreover, the uniqueness of solutions to \eqref{aux} follows as well. Indeed, set $u_1 = S(v_1)$ and $u_2 = S(v_2)$. Subtracting equations and multiplying it by $u_1-u_2$, we have
 \begin{align*}
  \lefteqn{
 \dfrac{\lambda}2 \dfrac{\d}{\d t} \|u_1(t)-u_2(t)\|_H^2 + \left(\B(u_1-u_2)(t), u_1(t)-u_2(t)\right)_H
  }\\
  &\quad + \left(\partial \varphi(u_1(t))- \partial \varphi(u_2(t)),u_1(t)-u_2(t)\right)_H\\
  &= - \left( B(v_1(t)) - B(v_2(t)), u_1(t)-u_2(t) \right)_H
 \end{align*}
for a.e.~$t \in (0,T)$. Integrating both sides over $(0,t)$ and using the monotonicity of $\B$ and $\partial \varphi$, we find that
\begin{align*}
\dfrac \lambda 2 \|u_1(t)-u_2(t)\|_H^2
 &\leq L_B \int^t_0 \|v_1(s)-v_2(s)\|_H \|u_1(s)-u_2(s)\|_H \, \d s\\
 &\leq L_B \|v_1-v_2\|_\mathcal{X} \|u_1-u_2\|_\mathcal{X} \int^t_0 \e^{2\omega s} \, \d s\\
 &\leq \dfrac{L_B}{2\omega} \e^{2\omega t}\|v_1-v_2\|_\mathcal{X} \|u_1-u_2\|_\mathcal{X} \quad \mbox{ for } \ t \in [0,T],
 \end{align*}
 which yields
 $$
 \|u_1-u_2\|_\mathcal{X} \leq \dfrac{L_B}{\omega \lambda} \|v_1-v_2\|_\mathcal{X}.
 $$
 Therefore choosing $\omega > L_B/\lambda$ and setting $\kappa := L_B/(\omega\lambda) \in (0,1)$, we obtain \eqref{contraction}. Thus we conclude that

\begin{proposition}
Let $T \in (0,\infty)$ and $\lambda \in (0,1)$ be fixed. For every $f \in L^2(0,T;H)$ and $u_0 \in D(\varphi)$, the Cauchy problem \eqref{aux} admits a unique strong solution $u \in W^{1,2}(0,T;H) \cap D(\partial \Phi)$ satisfying $u(0) = u_0$.
 \end{proposition}

Therefore for each $T \in (0,\infty)$, $\lambda \in (0,1)$, $u_0 \in D(\varphi)$ and $f \in L^2(0,T;H)$, the Cauchy problem \eqref{ee-aprx2} admits a unique strong solution $u_\lambda \in W^{1,2}(0,T;H)$ on $[0,T]$. Concerning the case where $T = \infty$, as in the proof of Theorem \ref{T:Lip}, one can verify that for each $\lambda \in (0,1)$, $u_0 \in D(\varphi)$ and $f \in L^2_{\rm loc}([0,\infty);H)$, the Cauchy problem \eqref{ee-aprx2} also admits a unique strong solution $u_\lambda \in W^{1,2}_{\rm loc}([0,\infty);H)$ on $[0,\infty)$.

\bibliographystyle{plain}
\bibliography{bibliography}

\end{document}